\documentclass[11pt,a4paper,oneside,DIV=12,headings=small,abstract=on]{scrartcl}
\usepackage[noblocks]{authblk}

\usepackage{amssymb,amsmath,amsfonts,amsthm,mathtools}
\usepackage{enumerate,expdlist,tikz,float}
\usetikzlibrary{patterns}
\usepackage[]{xcolor}
\usepackage[english]{babel}
\usepackage[T1]{fontenc}
\usepackage[utf8]{inputenc}
    \colorlet{annuli}{black!90}
\DeclareMathOperator{\Dom}{{\mathcal D}}
\DeclarePairedDelimiter{\abs}{|}{|}
\DeclarePairedDelimiter{\norm}{\lVert}{\rVert}
\newcommand{\euler}{\mathrm{e}}

\newcommand{\sfuc}{\mathrm{uc}}
\newcommand{\drm}{\mathrm{d}}
\newcommand{\RR}{\mathbb{R}}
\newcommand{\CC}{\mathbb{C}}
\newcommand{\NN}{\mathbb{N}}
\newcommand{\ZZ}{\mathbb{Z}}

\newcommand{\diam}{\operatorname{diam}}
\newcommand{\dist}{\operatorname{dist}}
\newcommand{\ran}{\operatorname{Ran}}
\newcommand{\ess}{\mathrm{ess}}
\newcommand{\equiset}{S_{\delta , Z}}
\newcommand{\equiop}{W_{\delta , Z}}
\newcommand{\Q}{\mathcal{Q}_{\Gamma,Z}}
\newcommand{\cD}{\mathcal{D}}
\newcommand{\cH}{\mathcal{H}}
\newcommand{\cK}{\mathcal{K}}
\newcommand{\cM}{\mathcal{M}}
\newcommand{\cL}{\mathcal{L}}
\newcommand{\fM}{\mathfrak{M}}
\newcommand{\E}{\mathsf{E}} 
\usepackage{mathtools}
\DeclareMathOperator{\funs}{s}

\newtheorem{theorem}{Theorem}[section]
\newtheorem{lemma}[theorem]{Lemma}
\newtheorem{proposition}[theorem]{Proposition}
\newtheorem{corollary}[theorem]{Corollary}
\newtheorem{hypothesis}[theorem]{Hypothesis}
\theoremstyle{definition}

\theoremstyle{remark}
\newtheorem{remark}[theorem]{Remark}
\newtheorem{example}[theorem]{Example}
\usepackage[linktocpage, pdftex]{hyperref}
  \definecolor{darkred}{rgb}{0.5,0,0}
  \definecolor{darkgreen}{rgb}{0,0.5,0}
  \definecolor{darkblue}{rgb}{0,0,0.5}
  \hypersetup{colorlinks,linkcolor=darkblue,filecolor=darkgreen,urlcolor=darkred,citecolor=darkblue}
  \hypersetup{
  pdftitle={Unique continuation and lifting of spectral band edges of Schr\"odinger operators on unbounded domains},
  pdfauthor={I. Nakic, M. T\"aufer, M. Tautenhahn, I. Veseli\'c},
  }
\makeatletter
\def\blfootnote{\xdef\@thefnmark{}\@footnotetext}
\makeatother
\begin{document}
%
%
%
%
%
%
\title
{Unique continuation and lifting of spectral band edges of Schr\"odinger operators on unbounded domains}

\author[1]{Ivica Naki\'c}
\affil[ ]{}
\author[2]{Matthias T\"aufer}
\author[3]{Martin Tautenhahn}
\author[4]{Ivan Veseli\'c}
\publishers{\vspace{0ex} \normalsize (With an Appendix by Albrecht Seelmann$^4$)}

\date{\vspace{-5ex}}
%
\maketitle
\begin{abstract}
We prove and apply two theorems: First, a quantitative, scale-free unique continuation
estimate for functions in a spectral subspace of a Schr\"odinger operator on a
bounded or unbounded domain; second, a perturbation and lifting estimate for edges of the essential spectrum  of a self-adjoint operator under a semi-definite perturbation. These two results are combined to obtain lower and upper Lipschitz bounds on the function parametrizing locally
a chosen edge of the essential spectrum of a Schr\"odinger operator
in dependence of a coupling constant.
Analogous estimates for eigenvalues, possibly in gaps of the essential spectrum,
are exhibited as well.
\end{abstract}
\footnotetext[1]{University of Zagreb, Faculty of Science, Department of Mathematics, Croatia}
\footnotetext[2]{Queen Mary University of London, School of Mathematical Sciences, United Kingdom}
\footnotetext[3]{Technische Universit\"at Chemnitz, Fakult\"at f\"ur Mathematik, Germany}
\footnotetext[4]{Technische Universit\"at Dortmund, Fakult\"at f\"ur Mathematik, Germany}
\blfootnote{Keywords: Unique continuation, uncertainty principle, semi-definite perturbation, spectral edges, lifting estimates, spectral engineering; MSC classes: 35Pxx, 35J10, 35B05, 35B60, 81Q10}
\tableofcontents

%
%
%
%
\section{Introduction}
There are two types of results in the paper, one from the realm of the analysis of partial differential equations and the other from operator theory.
On the one hand we prove a quantitative and scale-free unique continuation principle for Schr\"o\-dinger operators in unbounded domains.
On the other hand we establish lower bounds on the sensitivity of spectra of self-adjoint operators under perturbations.
Finally, we combine these results and obtain lower bounds on the lifting of eigenvalues and edges of essential spectrum of Schr\"odinger operators under certain non-negative perturbations.
\par
The classical notion of unique continuation for elliptic differential expressions $L$ refers to the fact that if a solution $u$ of $Lu = 0$ in a connected, open $\Gamma \subset \RR^d$ vanishes on a non-empty open subset, then $u$ is identically zero. In this paper, we study a quantitative and multi-scale variant of this property. Let
\begin{equation}\label{eq:Gamma}
\Gamma = \bigtimes_{k =1}^d (\alpha_i , \beta_i) , \quad \alpha_i ,\beta_i \in \RR\cup \{\pm \infty\}
\end{equation}
be a generalized rectangle, and $S_\delta \subset \Gamma$ be a certain equidistributed arrangement of balls of radius $\delta$. We refer to Fig.~\ref{fig:equisequence} for an illustration of such a set, and to Sect.~\ref{ssec:model} for a precise definition.
\begin{figure}[ht] \centering
\begin{tikzpicture}
  \pgfmathsetseed{{\number\pdfrandomseed}}
  \pgfmathsetmacro{\d}{0.25};
  \pgfmathsetmacro{\e}{0.5-\d};
  \foreach \x in {0.5,1.5,...,4.5}{
    \foreach \y in {0.5,1.5,...,4.5}{
    \filldraw[annuli, fill opacity = 0.3] (\x+rand*\e,\y+rand*\e) circle (\d);
    }
  }
  \foreach \y in {0,1,2,3,4,5}{
    \draw[dotted,thick] (\y,-0.3) --(\y,5.3);
    \draw[dotted,thick] (-0.3,\y) --(5.3,\y);
  }
\end{tikzpicture}
\caption{Illustration of a possible arrangement of the set $S_\delta \subset \Gamma$.}
\label{fig:equisequence}
\end{figure}
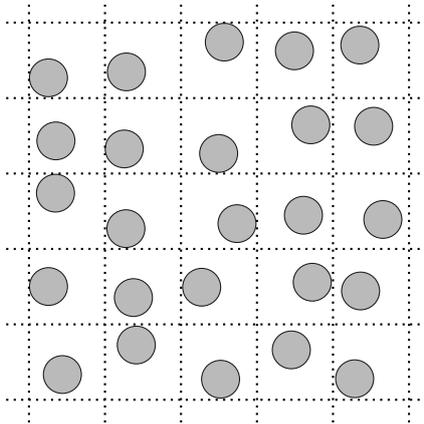
Furthermore, let $V \in L^\infty (\Gamma)$ be real-valued, and define the self-adjoint operator $H = -\Delta + V$ in $L^2 (\Gamma)$.
Our first main result, spelled out in Section~\ref{sec:UCP}, is an estimate of the form
\begin{equation}\label{eq:UCP-introduction}
 C \lVert u \rVert_{L^2 (\Gamma)}^2 \leq \lVert u \rVert^2_{L^2 (S_\delta)}
 \quad\text{with}\quad
 C = \delta^{N (1 + \lVert V \rVert_\infty^{2/3} + \sqrt{E})}.
\end{equation}
It holds for all $E \geq 0$ and all $u$ in the spectral subspace of $H$ up to energy $E$ with a uniform, merely dimension-dependent constant $N > 0$.
It is important to note that the constant $C$ does neither depend on the set $\Gamma$ nor on the choice of the equidistributed set $S_\delta$.
Furthermore, it depends on $V$ only via its $L^\infty$-norm.
\par
Results of the type \eqref{eq:UCP-introduction} have been obtained before in the case where $\Gamma$ is a cube $\Lambda_L$ of side length $L$.
In this case the operator $H$ is lower semi-bounded with purely discrete spectrum and all functions in its spectral subspace up to energy $E$ turn out to be finite linear combinations of eigenfunctions.
This is the typical situation encountered when studying random Schr\"odinger operators, Anderson localization, and Wegner estimates. However, if $\Gamma \subset \RR^d$ is an unbounded set, which is allowed in our framework,
the bulk of the spectrum of $H$ will in general consist of essential spectrum,
and eigenfunctions, if any exist, might span only a subspace. Consequently, by allowing a more general geometry (unbounded sets $\Gamma$) we are also allowing more general spectral situations
(spectral subspaces without an orthonormal basis consisting of eigenvectors).
In order to prove Ineq.~\eqref{eq:UCP-introduction} in the desired generality, we replace the expansion with respect to eigenfunctions by an appropriate formulation in terms of spectral calculus.

While it is not possible to give an exhaustive summary on  the history of unique continuation principles for Schr\"odinger operators here,
there is a number of papers which motivate directly and are the starting point of the present study. We describe them now briefly.
If $\Gamma = \Lambda_L$ with $L \in \NN$, and $u$ is in the domain of the Dirichlet Laplacian on $\Lambda_L$
satisfying $\lvert \Delta u \rvert \leq \lvert V u \rvert$ almost everywhere on $\Lambda_L$,
special cases of  Ineq.~\eqref{eq:UCP-introduction} have been obtained in \cite{Rojas-MolinaV-13} using local quantitative unique continuation estimates very similar to those developed in \cite{GerminetK-13b}  and \cite{BourgainK-13}.
Here, the condition $\lvert \Delta u \rvert \leq \lvert V u \rvert$ is in particular satisfied if $u$ is a single eigenfunction of a Schr\"odinger operator in $L^2 (\Lambda_L)$.
The case where $\Gamma = \Lambda_L$ and $u$ is in the spectral subspace of an interval $I \subset \RR$, has first been studied in \cite{Klein-13}
under the assumption that $I = [E-\epsilon , E + \epsilon]$ for fixed, but sufficiently small $\epsilon > 0$. The general case where $I = (-\infty , E]$ has been announced in \cite{NakicTTV-15}, while full proofs were given in \cite{NakicTTV-18}.
An extension of this result to infinite linear combinations of eigenfunctions with exponentially decaying coefficients has been given in \cite{TaeuferT-17}.
Special cases of an estimate of the type \eqref{eq:UCP-introduction} on unbounded domains have been presented first in \cite{TautenhahnV-16}.
\par
In Section \ref{sec:perturb} we study lower bounds on the displacement of eigenvalues and spectral edges
of general self-adjoint operators under the influence of non-negative perturbations
which are positive on an appropriate spectral subspace.
On one hand, these  abstract results are of interest in themselves,
on the other the required positivity on spectral subspaces is, by Ineq.~\eqref{eq:UCP-introduction},
in particular satisfied for Schr\"odinger operators and non-negative perturbations potentials which are positive on a superset of $S_\delta$ as in Fig.~\ref{fig:equisequence}.
\par
To formulate this more precisely, let $E\in \RR$, $A$ a self-adjoint operator and $B$ a self-adjoint, non-negative, bounded perturbation, satisfying
\begin{equation}
 \label{eq:2}
 B \geq \kappa >0
 \quad
 \text{on}
 \quad
 \ran P_{A + B}(E),
\end{equation}
where $\{ P_{A+B}(E) \colon E \in \RR \}$ is the resolution of identity for the operator $A+B$.
If $A$ is lower semi-bounded with purely discrete spectrum then
a standard perturbation argument shows that the minimax principle for eigenvalues yields
\[
 \lambda_k(A + B) \geq \lambda_k(A) + \kappa
\]
as long as $\lambda_k(A + B) \leq E$, where $\lambda_k$ denotes the $k$-th eigenvalue of an operator, counted from below including multiplicities.
For convenience of the reader and in order to illustrate the underlying mechanism, we give the precise statement and proof in Lemma~\ref{lem:bottom} below.
Now, if the operator $A$ exhibits essential spectrum as well, the corresponding situation is more challenging.
By an adaptation of the argument for eigenvalues, it is still rather straightforward to deduce an analogous lower bound on the sensitivity of the bottom of the essential spectrum with respect to perturbations, see Lemma~\ref{lem:bottom} below.
\par
The situation concerning the movement of other edges of the essential spectrum and eigenvalues inside gaps of the essential spectrum is far more intricate.
In each gap of the essential spectrum, the operator $A$ may have a finite or countably infinite set of isolated eigenvalues of finite multiplicity with possible accumulation points at the upper and lower edge of the gap.
Therefore, the notion of the $k$-th eigenvalue inside a gap might become meaningless.
Hence, for a self-adjoint operator $A$ we pick a reference point $\gamma \in \RR \cap \rho(A)$ and denote by $\lambda^\leftarrow_{k,\gamma}$ (respectively, $\lambda^\rightarrow_{k,\gamma}$)  the $k$-th eigenvalue to the left (respectively to the right) of $\gamma$, counting multiplicities.
Under a similar assumption as Ineq.~\eqref{eq:2}, we deduce in Theorem~\ref{thm:gaps_right} for perturbations $B$ with norm not exceeding a critical value that
\begin{equation}
 \label{eq:eigenvalue_moving_intro}
 \lambda_{k, \gamma}^{\rightarrow}(A + B)
 \geq
 \lambda_{k, \gamma}^{\rightarrow}(A) + \kappa .
\end{equation}
Indeed, the mentioned assumption on $\|B\|$ is necessary in order to ensure that the spectral components above and below $\gamma$ do not mix, and is optimal in that regard.
We also show an analogous estimate for edges of the essential spectrum, see Theorem~\ref{thm:gaps_left} below. The proofs rely on variational principles which are valid  inside gaps of the essential spectrum, and which have been developed in other contexts,
see, e.g., \cite{GriesemerLS-99}, \cite{MorozovM-15}, \cite{LangerS-16}, and the references given there. In fact our proofs employ two different variational principles, one for spectrum below $\gamma$ and one for spectrum above $\gamma$.
The first of them has recently been proved in \cite{LangerS-16}.
The second one is a combination of the variational principle
in \cite{GriesemerLS-99}
and the subspace perturbation theory of \cite{Seelmann-17-arxiv}.
This variant of the variational principle is formulated and proved in the appendix by A.~Seelmann of this paper.
We need this variant, because the direct application of \cite{GriesemerLS-99} or \cite{MorozovM-15}, which are formulated for quadratic forms,
does not seem to yield the  optimal critical value for the norm of the perturbation $\lVert B \rVert$,
which we were aiming for.
Apart from this, we feel that the variational principle presented in the appendix relies on criteria which are particularly natural and easy to check.
In fact, in our application the criteria are ensured by the
Davis-Kahan $\sin 2 \Theta$ theorem on the perturbation of spectral projections.
\par
Moreover, one might wonder whether one can avoid to invoke two different variational principles when considering eigenvalues below and above $\gamma$, respectively.
One would naturally try to transform one variational principle into the other by using the
sign-flip  $A-\gamma \mapsto -A+\gamma$.
However, this does not work directly, because assumption \eqref{eq:2} is not symmetric under the sign flip.
In particular, while for lower semi-bounded operators $A$  the operator $(A+B)P_{A+B}(E)$ is bounded, $(A+B)(1-P_{A+B}(E))$ is not, for arbitrary $E\in \RR$.
\par
We also prove monotonicity of eigenvalues
$\lambda_{k, \gamma}(A + C) \leq \lambda_{k, \gamma}(A + B)$
(inside gaps of the essential spectrum)
with respect to two perturbation operators $0\leq C\leq B$. Here $\lambda_{k, \gamma}$ may denote either $\lambda^{\rightarrow}_{k, \gamma}$ or $\lambda^{\leftarrow}_{k, \gamma}$.
\par
Perturbation of eigenvalues inside gaps has been studied in a number of papers, but
the questions raised there are different to those which we study.
Mostly the number of eigenvalues inside an essential spectral gap is studied, see
\cite{HundertmarkS-08} for an abstract result, oftentimes in the asymptotic regime of
a coupling constant, see \cite{Birman-91}.
For operators of Schr\"odinger type or of divergence form there are numerous more precise results
of this genre, in particular those originating from Birman's school, see for instance
\cite{BirmanL-94,BirmanW-94,Birman-95,Birman-96,Birman-97,BirmanP-97,BirmanS-10,Suslina-03}.

\par

We now turn to sensitivity of the edges of the essential spectrum. Lemma~\ref{lem:lipschitz_left_edge} formulates a folklore wisdom that the edges of the essential spectrum define locally Lipschitz continuous functions with respect to the perturbation.
See \cite{Veselic-K-08} and \cite{Seelmann-14-diss} for some background on this topic.
One of the main results of the present paper is to complement this estimate by giving lower bounds on the lifting of the spectral edges.
In particular, under a slightly stronger assumption than \eqref{eq:2} we show in
Theorem~\ref{thm:essential_left} and Theorem~\ref{thm:essential_right} that the edge of the essential spectrum moves at least by $\kappa$ for sufficiently small perturbations.
Basically, the position $f(t)$ of the specified edge of $\sigma_{\mathrm{ess}}(A+tB)$ satisfies
\begin{equation}\label{eq:f-bound}
f(0) + t \kappa \leq f(t) \leq f(0) + t \Vert B\Vert
\end{equation}
as long as $t \in [0,t_0)$ with $\lVert B \rVert t_0 $ equal to the length of the gap in the essential spectrum above $f(0)$.
Again, our smallness condition on the perturbation $B$, or the bound on $t_0$, respectively, is optimal.
The proof relies on a successive application of the perturbation result \eqref{eq:eigenvalue_moving_intro}, an iterative renumbering, and a limiting procedure.
\par
In Section~\ref{sec:application}, we combine the results of the previous two sections.
For $A$ we choose a Schr\"odinger operator in $L^2 (\Gamma)$ with a bounded potential $V$, and
for $B$ a multiplication operator satisfying $B \geq \vartheta \mathbf{1}_{S_{\delta}}$ for some $\vartheta > 0$.
Then the quantitative unique continuation principle from \eqref{eq:UCP-introduction} ensures that our perturbation results can be applied. Hence, for such non-negative perturbations of Schr\"odinger operators,
we obtain quantitative estimates on the lifting of  eigenvalues and (essential) spectral edges.
Let us stress that the number $\kappa$ is explicitly given in terms of $\vartheta$, $\delta$, $\lVert V \rVert_\infty$, $\lVert B \rVert$, and the upper energy cutoff $E$.
Furthermore, the estimates are scale-free and uniform over geometric configurations in the sense that they neither depend on the set $\Gamma$ nor on the particular realization of the equidistributed arrangement of balls $S_\delta$ on which the perturbation is supported.
Let us give a specific example where our theorem applies.
\begin{example}
  Let $\cL_1$ and $\cL_2\subset \RR^d$ be two incommensurate $d$-dimensional lattices,
  $V_j \in L^\infty(\RR^d)$ be real-valued
  $\cL_j$-periodic, for $j=1$ respectively  $j=2$, and $H=-\Delta +V_1$.
  Then $H$ is self-adjoint with purely absolutely continuous band spectrum, and the norm $\lVert V_j\rVert$ of the multiplication operator
  equals the sup-norm $\lVert V_j\rVert_\infty$ of the function.
  Assume that $(0,b)$ is a spectral gap of $H$, i.e.\ $0,b\in \sigma(H)=\sigma_\ess(H)=\sigma_{\mathrm{ac}}(H)$ and $\sigma(H)\cap (a,b)=\emptyset$.
  For $t\in \RR$ set
  $f(t) = \sup \left( \sigma_{\mathrm{ess}}(H + t V_2) \cap (-\infty, b-\vert t\vert \lVert V_2 \rVert) \right) $. Then $f(0)=0$ and for sufficiently small values of $\vert t\vert$
  the corresponding edge of $\sigma_\ess(H+tV_2)$ is given by $f(t)$ and satisfies
  $-\vert t\vert \lVert V_2 \rVert \leq f(t)\leq \vert t\vert \lVert V_2 \rVert$. More precisely, this holds true for all $t \in (- b / (2\lVert V_2 \rVert) ,b / (2\lVert V_2 \rVert))$.

  In the case that there is an open $O\subset \RR^d$ such that $V_2\geq \mathbf{1}_{O}$
  one can simplify the definition of $f$ to
  \[
  f\colon \big[0, b/\lVert V_2 \rVert\big) \to \RR, \quad
  f(t) = \sup \left( \sigma_{\mathrm{ess}}(H + t V_2) \cap (-\infty, b) \right).
  \]
  Then we have for all $t \in [0, b / \lVert V_2 \rVert )$ and
  $\epsilon \in [0, b / \lVert V_2 \rVert - t)$ the two-sided Lipschitz bound
  \[
  \epsilon \kappa \leq f(t+\epsilon) -f(t)\leq \epsilon \lVert V_2 \rVert
  \]
  where $\kappa$ is as above and thus depends only on $d$, $O$,
  $\lVert V_1 \rVert$, $\lVert V_2 \rVert$, and  $b$.

  If the two lattices $\cL_1$ and $\cL_2$, happen to be commensurate, i.e.
  $\cL_1\cap \cL_2$ is a $d$-dimensional lattice as well, then the result of course holds as well.
  However, in this simpler situation one can resort to a Floquet-decomposition to reduce the problem to
  operators with purely discrete spectrum. In the general situation we consider here, this is not the case.
\end{example}
Using the estimates in Section  \ref{sec:perturb}
and Section \ref{sec:application}
one can perform spectral engineering using (possibly non-periodic) perturbation potentials in the following sense: 
If there is an energy $E$ near  the boundary of the essential spectrum, which is in the resolvent set, 
one can identify appropriate perturbation potentials such that the essential spectrum of the perturbed Schr\"odinger covers the prescribed energy. 
Conversely, if the energy $E$ is in the spectrum, appropriately chosen additive perturbations yield an operator where $E$ is in the resolvent set.
\par 

The two-sided Lipschitz estimate on the movement of spectral band edges of Schr\"odinger operators
provides a key tool
to prove Anderson localization near edges of the essential spectrum of some
Schr\"odinger operator $H$ on $L^2(\RR^d)$ under the influence of a random potential
$W=W_\omega$. Indeed, in the case of a $H$ with periodic potential
and a perturbation $W_\omega$ of alloy-type this has been performed
in \cite{KirschSS-98a}.
The starting point of their analysis is Ineq.~(2.2) in  \cite[Theorem~2.2]{KirschSS-98a}
which is a special case of bound \eqref{eq:f-bound} above.
However, \cite[Theorem~2.2]{KirschSS-98a} considers only a periodic situation.
Due to this assumption one can use Floquet-Bloch theory to reduce the spectral perturbation problem
to operators with purely discrete spectrum and then apply classical perturbation theory for isolated eigenvalues.
In the general situation which we are treating this is no longer possible
and we have to study the essential spectrum directly.
We defer the precise formulation and proof of the indicated claim on Anderson localization near spectral band edges to a sequel paper.

There is another important application of the unique continuation results of Section \ref{sec:UCP},
namely control theory of the heat equation on unbounded domains.
While control theory on bounded domains has been studied for decades, unbounded domains became a focus of interest rather recently,
see for instance \cite{Miller-05a,Miller-05b}, \cite{GonzalezT-07}, \cite{Barbu-14}, \cite{Zhang-16},
\cite{LeRousseauM-16}, \cite{WangWZZ}, \cite{EgidiV-17-arxiv}, \cite{BeauchardPS-18}
and the references therein.
In a sequel project \cite{NakicTTV-20} we have established explicit estimates on the cost of null-controllability of the heat equation for
Schr\"odinger semigroups on unbounded domains $\Gamma$ of the type \eqref{eq:Gamma}.
They recover, improve, and unify several results obtained earlier.

%
%
%
%
\section{Quantitative unique continuation principle} \label{sec:UCP}
\subsection{Notation and main result}
\label{ssec:model}
Whenever we consider in this paper Schr\"odinger operators we consider the following geometric situation.
Let $d \in \NN$. For $G > 0$ we say that a set $\Gamma \subset \RR^d$ is $G$-\emph{admissible}, if there exist $\alpha_i ,\beta_i \in \RR\cup \{\pm \infty\}$ with $\beta_i - \alpha_i \geq G$ for $i \in \{1,\ldots,d\}$, such that
\[
 \Gamma = \bigtimes_{i =1}^d (\alpha_i , \beta_i) .
\]
This implies that there is a $\xi \in \RR^d$ such that $(-G/2 , G/2)^d + \xi$ is contained in $\Gamma$. For $\delta > 0$ and such $\xi \in \RR^d$ we say that a sequence $Z = (z_j)_{j \in (G\ZZ)^d + \xi} \subset \RR^d$ is \emph{$(G,\delta , \xi)$-equidistributed}, if
 \[
  \forall j \in (G\ZZ)^d + \xi \colon \quad  B(z_j , \delta) \subset (-G/2 , G/2)^d + j.
\]
Corresponding to a $G$-admissible $\Gamma \subset \RR^d$ and a $(G,\delta , \xi)$-equidistributed sequence $Z$, we define
\[
\equiset = \bigcup_{j \in (G\ZZ)^d + \xi} B(z_j , \delta) \cap \Gamma.
\]
Note that by a global shift of the coordinate system, we may and will henceforth assume without loss of generality
that $\xi=0\in \RR^d$. We will call a  $(G,\delta , 0)$-equidistributed sequence simply
$(G,\delta)$-equidistributed.
\par
For a $G$-admissible set $\Gamma$ and real-valued $V \in L^\infty(\Gamma)$, we define the self-adjoint operator $H$ on $L^2 (\Gamma)$ as
\begin{equation*}
 H = -\Delta + V
\end{equation*}
with Dirichlet or Neumann boundary conditions.
\smallskip

For a set $S \subset \RR^d$, we denote by $\mathbf{1}_S$ its indicator function.
For a self-adjoint operator $A$ we denote by $P_A(E):=\mathbf{1}_{(-\infty, E]}(A)$ the spectral projector
of $A$ associated to the interval $(-\infty, E]$ and by $P_A$ the spectral measure of $A$.
We use the notation $P_A^\perp (E) = \mathbf{1}_{(E , \infty)}(A) $ for the projection onto the orthogonal complement of $\ran P_A (E)$.
For details concerning spectral calculus we refer to \cite{ReedS-78} or \cite{Schmuedgen-12}.
The main result of this section is the following theorem.
\begin{theorem} \label{thm:result1}
There is $N>0$ depending only on the dimension, such that for all
$1$-admissible $\Gamma \subset \RR^d$, all $\delta \in (0,1/2)$, all $(1,\delta)$-equidistributed sequences $Z$, all real-valued $V \in L^\infty (\Gamma)$, all $E \in \RR$, and all $\psi \in \ran P_H (E)$ we have
\begin{equation*}
\label{eq:result1}
\lVert \psi \rVert_{L^2 (\equiset)}^2
\geq C_\sfuc \lVert \psi \rVert_{L^2 (\Gamma)}^2 ,
\quad\text{where}\quad
C_\sfuc =  \delta^{N \bigl(1 + \lVert V \rVert_\infty^{2/3} + \sqrt{\max \{ 0, E \}} \bigr)}.
\end{equation*}
\end{theorem}
By scaling and optimizing over a spectral shift parameter we obtain
\begin{corollary} \label{cor:scaling}
There is $N>0$ depending only on the dimension, such that for all
$G>0$, all $G$-admissible $\Gamma \subset\RR^d$, all $\delta \in (0,G/2)$, all $(G,\delta)$-equidistributed sequences $Z$, all real-valued $V \in L^\infty (\Gamma)$, all $E \in \RR$, and all $\psi \in \ran P_H (E)$ we have
\begin{equation*}
\lVert \psi \rVert_{L^2 (\equiset)}^2
\geq C_\sfuc^{(G)} \lVert \psi \rVert_{L^2 (\Gamma)}^2 ,
\quad\text{where}\quad
 C_\sfuc^{(G)} = \sup_{\lambda \in \RR}
 \left(\frac{\delta}{G}\right)^{N \bigl(1 + G^{4/3}\lVert V-\lambda \rVert_\infty^{2/3} + G\sqrt{(E-\lambda)_+} )}
\end{equation*}
and $t_+:=\max\{0,t\} $ for $t \in \RR$.
\end{corollary}
\begin{remark}
Choices of the parameter $\lambda$ which lead to a value for $ C_\sfuc^{(G)}$
which is not too far off from the optimal are for instance
$\lambda=\inf \sigma(H)$, $\lambda=\inf_x V(x)$ or  $\lambda=(\inf_x V(x)+\sup_x V(x))/2$.
\par
We emphasize that the constant $C_\sfuc^{(G)}$ neither depends on the set $\Gamma$ nor on the choice of the $(G,\delta)$-equidistributed sequence $Z$, and that it depends on the potential $V$ only via its $L^\infty$-norm.
\end{remark}
\begin{corollary}\label{cor:uncertainty}
Let $\equiop=\mathbf{1}_{\equiset}$ be the operator of multiplication by the characteristic function of the set $\equiset$.
Under the assumptions of Corollary~\ref{cor:scaling} we have
\begin{equation*}\label{eq:abstractUCP}
P_H (E) \equiop P_H (E)
\geq C_\sfuc^{(G)}
P_H (E)
\end{equation*}
in the sense of quadratic forms.
\end{corollary}
The proofs of Theorem \ref{thm:result1} and Corollary \ref{cor:scaling} are given in subsection \ref{ssec:proofs}.
%
%
%
%
%
%
%
%
\subsection{Extension to ghost dimension and spectral function}\label{ssec:ghost}
For the application of an appropriate unique continuation result
it will be necessary to extend the operator $H=-\Delta +V$ to a differential expression acting on a $(d+1)$-dimensional domain. We discuss now this extension in an abstract framework.

 Let $\cH$ be a Hilbert space and $A$ a self-adjoint operator on $\cH$ with domain $\cD (A)$.
Note that $A$ need not be semi-bounded.
We define the family of self-adjoint operators $(\mathcal{F}_t )_{t \in \RR}$ on $\cH$ as
\[
 \mathcal{F}_t
 =
 \int_{- \infty}^\infty s_t(\lambda) \mathrm{d} P_A(\lambda)
 \quad
 \text{where}
 \quad
 s_t(\lambda)=\begin{cases}
        \sinh(\sqrt{\lambda} t)/\sqrt{\lambda}, & \lambda > 0 ,\\
        t, & \lambda=0,\\
        \sin(\sqrt{- \lambda} t)/\sqrt{- \lambda}, & \lambda <0.
\end{cases}
\]
The operators $\mathcal{F}_t$ are self-adjoint operators with $\ran P_A ([a,b]) \subset \cD (\mathcal{F}_t)$ for $-\infty < a < b < \infty$, where $\cD (\mathcal{F}_t)$ denotes the domain of $\mathcal{F}_t$.
For $\psi \in \ran P_A ([a,b])$ we define the function $\Psi \colon \RR \to \cH$ as
\[
 \Psi(t)
 =
 \mathcal{F}_t \psi .
\]
For $T > 0$ we define the operator $\hat A$ on $L^2 ((-T,T) ; \cH)$ with $\cD (\hat A) = \{\Phi \colon t \mapsto A (\Phi (t)) - (\partial_t^2 \Phi) (t) \in L^2 ((-T,T) ; \cH)\}$
by
\[
 (\hat A \Phi) (t) =   A (\Phi (t)) - (\partial_t^2 \Phi) (t) .
\]
where $\partial_t^2 \Phi$ denotes the second $\cH$-derivative with respect to $t$.
\begin{lemma} \label{lemma:properties_Psi}
For all $a,b \in \RR$ with $a < b$ and all $\psi \in \ran P_A ([a,b])$ we have:
\begin{enumerate}[(i)]
 \item The map $\RR \ni t \mapsto \Psi (t) \in \cH$ is infinitely $\cH$-differentiable.
 In particular,
 \begin{equation}
 \label{eq:property_F_1}
  (\partial_t \Psi) (0)
  = \psi.
 \end{equation}
 \item
 For all $T > 0$ we have $\Psi \in \cD (\hat A)$ and
 \begin{equation}
 \label{eq:property_F_2}
 \hat A \Psi = 0 .
 \end{equation}
\end{enumerate}
\end{lemma}
\begin{remark} \label{rem:ext}
We apply the lemma in the case $\cH = L^2 (\Gamma)$ and $A = H$, $a = -\lVert V \rVert_\infty$, and $b > a$. Then it follows that
$L^2 ((-T,T);\cH) = L^2 (\Gamma \times (-T,T))$, i.e.\ $\Psi$ may be understood as a mapping from $\Gamma \times (-T,T)$ to $\CC$. We have
\[
 \hat A = -\Delta + \hat V
 \quad \text{on} \quad L^2 (\Gamma \times (-T,T))
\]
with corresponding boundary conditions on $(\partial \Gamma) \times (-T,T)$, where $\hat V (x,t) = V (x)$.
Since $\Psi \in \cD (\hat A)$ by Lemma~\ref{lemma:properties_Psi}, we find that $\Psi \in H^1 (\Gamma \times (-T,T))$.
Now we extend $\Psi$ and $\hat V$ from $\Gamma \times (-T,T)$ to $\Gamma^+ \times (-T,T)$ where
\[
 \Gamma^+ := \bigtimes_{i=1}^d (\alpha_i^+ , \beta^+_i ) ,
\]
with
\[
\alpha_i^+ = \frac{M (\alpha_i - \beta_i) + \alpha_i + \beta_i}{2} \quad
\beta_i^+ = \frac{M (\beta_i - \alpha_i) + \alpha_i + \beta_i}{2}
\]
if both $\alpha_i$ and $\beta_i$ are finite, and $(\alpha_i^+ , \beta_i^+) = \RR$ otherwise.
Here $M$ is the least power of 3 which is larger than $9 \euler \sqrt{d}$.
The construction of the extension will depend on the boundary conditions, see\ \cite{NakicTTV-18}. In the case of
\begin{itemize}
 \item Dirichlet boundary conditions we extend $\hat V$ for each direction $i \in \{1,\ldots , d\}$ for which  one of $\alpha_i$ and $\beta_i$ is finite by symmetric reflections with respect to the boundary, and $\Psi$ by antisymmetric reflections. Then we iterate this procedure $\log_3 M$ times.
 \item Neumann boundary conditions we extend both $\hat V$ and $\Psi$ in the same way by symmetric reflections.
\end{itemize}
We will use the same symbol for the extended $\hat V$ and $\Psi$. Note that $\hat V$ takes values in $[-\lVert V \rVert_\infty, \lVert V \rVert_\infty]$. By Lemma~\ref{lemma:properties_Psi}
\[
 \hat A \Psi = (-\Delta + \hat V)\Psi = 0 \quad \text{on}\quad \Gamma \times (-T,T).
\]
By construction, we have
\begin{equation} \label{eq:eigenfunct_ext}
  (-\Delta + \hat V)\Psi = 0 \quad \text{on} \quad \Gamma^+ \times (-T,T),
\end{equation}
and (the extended) $\Psi$ satisfies the corresponding boundary conditions.
\end{remark}
\begin{proof}[Proof of Lemma~\ref{lemma:properties_Psi}]
In order to identify the $L^2(\cH)$-derivative
we calculate using dominated convergence theorem
\begin{multline*}
   \lim_{h \to 0}
   \biggl\lVert
   \int_a^b
   \left(
   \frac{s_{t+h}(\lambda) - s_t(\lambda)}{h} - \partial_t s_t(\lambda)
   \right) \mathrm{d} P_A(\lambda) \psi
   \biggr\rVert_{\cH}^2\\
   =
   \lim_{h \to 0}
   \int_a^b
   \biggl\lvert
   \frac{s_{t+h}(\lambda) - s_t(\lambda)}{h} - \partial_t s_t(\lambda)
   \biggr\rvert^2
   \mathrm{d} \langle \psi, P_A(\lambda) \psi \rangle
   =
   0.
  \end{multline*}
  The calculation for higher derivatives is analogous and we find for $k \in \NN_0$ that
  \[
  (\partial_t^k \Psi)(t) = \left( \int_{a}^b \partial_t^k s_t(\lambda) \mathrm{d} P_A(\lambda) \right) \psi \in \cH.
 \]
  \par
  Since $A$ is self-adjoint and $\mathcal{F}_t P_A ([a,b])$ is bounded, the operator $A \mathcal{F}_t P_A ([a,b])$ is closed.
  For part (ii) we infer from \cite[Theorem~5.9]{Schmuedgen-12} that
  \[
  A \mathcal{F}_t P_A ([a,b]) \psi
  =
  \overline{A \mathcal{F}_t P_A ([a,b])} \psi
  =
  \left(\int_a^b \lambda s_t(\lambda) \mathrm{d} P_A(\lambda) \right) \psi
  \in
  \cH ,
  \]
  which implies that $\psi \in \cD(A \mathcal{F}_t P_A ([a,b]))$.
  Hence $\mathcal{F}_t P_A ([a,b]) \psi = \Psi(t) \in \cD(A)$.
  We calculate using $\lambda s_t(\lambda) - \partial_t^2 s_t(\lambda) = 0$
  \[
   \int_{-T}^T \lVert A(\Psi(t)) - (\partial_t^2 \Psi)(t) \rVert_{\cH}^2 \mathrm{d}t
   =
   \int_{-T}^T
   \int_a^b
   \lvert \lambda s_t(\lambda) - \partial_t^2 s_t(\lambda) \rvert^2
    \mathrm{d} \lVert P_A(\lambda) \psi \rVert^2
   \mathrm{d}t
   =
   0
   \qedhere
  \]
\end{proof}
\subsection{Interpolation inequalities}
In this section we formulate two interpolation inequalities resulting from optimization of Carleman estimates.
They are generalizations of  Propositions~3.4 and 3.5 in \cite{NakicTTV-18}.
For the sake of the reader's convenience we use the notation from \cite{NakicTTV-18}
and define $X_T = \Gamma \times (-T,T)$ for $T > 0$ and $\RR^{d+1}_+ = \{ (x,t) \in \RR^{d} \times \RR \colon t \geq 0 \}$ for a $1$-admissible set $\Gamma \subset \RR^d$, and
\begin{align*}
 S_1 &= \biggl\{(x,t) \in \RR^{d} \times [0,1] \colon -t + \frac{t^2}{2} - \frac{\sum_{i=1}^d x_i^2}{4} > -\frac{\delta^2}{16}  \biggr\} \subset \RR^{d+1}_+  , \\
 S_3 &= \biggl\{(x,t) \in \RR^{d} \times [0,1] \colon -t + \frac{t^2}{2} - \frac{\sum_{i=1}^d x_i^2}{4} > -\frac{\delta^2}{4} \biggr\} \subset \RR^{d+1}_+ .
 \end{align*}
Note that in \cite{NakicTTV-18} an additional set $S_2$ is defined, which we do not need here.
Moreover, given a $1$-admissible set $\Gamma \subset \RR^d$ and a $(1,\delta)$-equidistributed sequence $Z$, we define
\[
 \Q
 =
 \left\{
  j \in \ZZ^d
  \colon
  (-1/2, 1/2)^d + j \subset \Gamma
 \right\}
\]
as well as
\[
U_{k,Z} = \bigcup_{j \in \Q} S_k (z_j) , \quad\text{for}\quad k \in \{1,3\}.
\]
Here $S_k (x) = S_k + x$ denotes the translate of the set $S_k$ by $x \in \RR^d$.
Recall that for $\psi \in \ran P_H(E)$, and $T > 0$, the function $\Psi \in H^1(\Gamma^+ \times (-T,T))$ is constructed as in Section~\ref{ssec:ghost}.
\begin{proposition} \label{prop:interpolation1}
For all $1$-admissible $\Gamma \subset \RR^d$, all $\delta \in (0,1/2)$, all $(1,\delta)$-equidistributed sequences $Z$, all real-valued $V \in L^\infty (\Gamma)$, all $E \in \RR$, and all $\psi \in \ran P_H (E)$ we have
\[
  \lVert  \Psi \rVert_{H^1 (U_{1,Z} )} \leq D_1 \lVert \psi \rVert_{L^2 (\equiset)}^{1/2} \lVert \Psi \rVert_{H^1 (U_{3,Z} )}^{1/2} ,
\]
where
\begin{equation*} 
   D_1 = \delta^{-N_1 (1 + \lVert V \rVert_\infty^{2/3})}
\end{equation*}
and $N_1 > 0$ is a constant depending only on the dimension.
\end{proposition}
If $\Gamma = \Lambda_L = (-L/2 , L/2)^d$ with $L \in \NN_{\mathrm{odd}}:=\{1,3,5,\ldots\}$, Proposition~\ref{prop:interpolation1} has been stated and proved in~\cite[Proposition~3.4]{NakicTTV-18}, where we note that therein, the analog of $\Psi$ is called $F$.
Let us now sketch how to prove Proposition~\ref{prop:interpolation1} in the case of $1$-admissible $\Gamma$. We will omit technical details given in \cite[Proposition~3.4]{NakicTTV-18}, and focus only on the necessary modifications. A comprehensive proof can be found in \cite{Taeufer-diss}.
\par
By Lemma~\ref{lemma:properties_Psi} and Remark~\ref{rem:ext} we have $\Delta \Psi = \hat V \Psi$ almost everywhere on $\Gamma \times (-T,T)$.
As a consequence, Carleman estimates as in~\cite{NakicTTV-18} lead to the inequality
\begin{equation}
 \label{eq:after_Carleman_interpolation_1}
 \lVert \Psi \rVert_{H^1(S_1(z_j))}^2
 \leq
 f_1(\beta)
 \lVert (\partial_t \Psi)(\cdot, 0) \rVert_{L^2(B_\delta(z_j))}^2
 +
 f_2(\beta)
 \lVert \Psi \rVert_{H^1(S_3(z_j))}^2
\end{equation}
for all $j \in \Q$.
Here, $\beta$ is an optimization parameter which can be chosen over some range $(\beta_1, \infty)$, the prefactor $f_1(\beta)$ is exponentially growing, and the prefactor $f_2(\beta)$ is exponentially decreasing with respect to $\beta$.
Furthermore, the dependence of $\beta_1$ and the prefactors on $\delta$ and $\lVert V \rVert_\infty$ is explicit, see~\cite{NakicTTV-18} for details.
Summing Ineq.~\eqref{eq:after_Carleman_interpolation_1} over $j \in \Q$
and using that for different $j \in \Q$ the $S_1(z_j)$ (or $S_3(z_j)$ or $B_\delta(z_j)$ respectively) are mutually disjoint, one arrives at
\[
 \lVert \Psi \rVert_{H^1(U_{1,Z})}^2
 \leq
 f_1(\beta)
 \lVert (\partial_t \Psi)(\cdot, 0) \rVert_{L^2(S_{\delta,Z})}^2
 +
 f_2(\beta)
 \lVert \Psi \rVert_{H^1(U_{3,Z})}^2.
\]
Now a special choice of $\beta$ ensures that the two summands become equal.
This turns the additive estimate into a multiplicative one.
Tracking the dependence of the resulting constant in terms of $\delta$ and $\lVert V \rVert_\infty$ then yields the statement of Proposition~\ref{prop:interpolation1}, see~\cite{NakicTTV-18} for details.
\par
We shall also need a second interpolation estimate.
\begin{proposition} \label{prop:interpolation2}

For all $1$-admissible $\Gamma \subset \RR^d$, all $\delta \in (0,1/2)$, all $(1,\delta)$-equidistributed sequences $Z$, all real-valued $V \in L^\infty (\Gamma)$, all $E \in \RR$, and all $\psi \in \ran P_H (E)$
we have
 \[
\lVert \Psi \rVert_{H^1 (X_1)}
\leq D_2 \lVert \Psi \rVert_{H^1 (U_{1,Z} )}^{\gamma} \lVert \Psi \rVert_{H^1 (\Gamma^+ \times (-R_3,R_3))}^{1- \gamma} ,
\]
where
\begin{equation*} 
R_3 = 9 \euler \sqrt{d}, \quad
\gamma = \left( \log_2 \left( \frac{8.5 \euler\sqrt{d}}{\frac{1}{2} - \frac{1}{8}\sqrt{16-\delta^2}} \right) \right)^{-1} , \quad
D_2^{1/\gamma} = \delta^{-N_2 (1+ \lVert V \rVert_\infty^{2/3})} ,
\end{equation*}
and $N_2>0$ is a constant depending only on the dimension.
\end{proposition}
%
If $\Gamma = \Lambda_L = (-L/2 , L/2)^d$ with $L \in \NN_{\mathrm{odd}}$, Proposition~\ref{prop:interpolation2} has been stated and proved in~\cite[Proposition~3.5]{NakicTTV-18} with a slightly different $\gamma$.
Again we sketch the proof of Proposition~\ref{prop:interpolation2} in the case of general $\Gamma$. We will omit technical details given in \cite[Proposition~3.5]{NakicTTV-18}, and focus only on the necessary modifications. A comprehensive proof can be found in \cite{Taeufer-diss}.
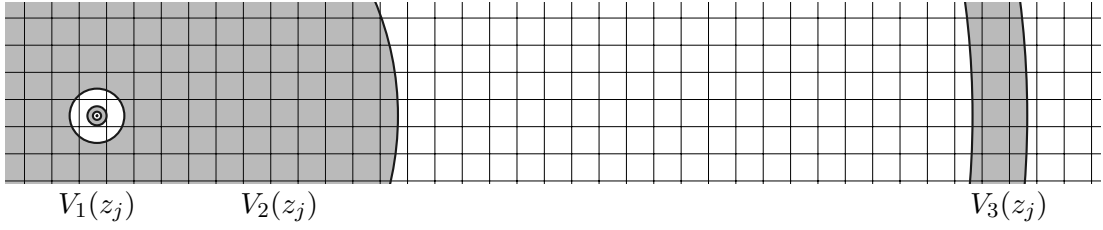
\begin{figure}[ht] \centering
\begin{tikzpicture}[scale = 1.2]
  \draw (0, -1) node {$V_1(z_j)$};
  \draw (2, -1) node {$V_2(z_j)$};
  \draw (10, -1) node {$V_3(z_j)$};
  \clip (-1,-.75) rectangle (11,1.25);
  \begin{scope}[scale = 0.3]
    \draw[shift={(.35,.6)},very thin] (-5,-5) grid (40,5);
    \draw[fill = annuli] (0,0) circle (1pt);
    \filldraw[thick, annuli, fill opacity = 0.3, even odd rule] (0,0) circle (0.15) -- (0,0) circle (0.35);
    \filldraw[thick, annuli, fill opacity = 0.3, even odd rule] (0,0) circle (1) -- (0,0) circle (11);
    \filldraw[thick, annuli, fill opacity = 0.3, even odd rule] (0,0) circle (32) -- (0,0) circle (34);
  \end{scope}
\end{tikzpicture}
\caption{Sketch of the annuli $V_1(z_j)$, $V_2(z_j)$, $V_3(z_j)$ in dimension $d = 2$}
\end{figure}

We define
\begin{align*}
r_1     &= \frac{1}{2} - \frac{1}{8}\sqrt{16-\delta^2}, & r_2 &=1,  &r_3 &= 8.5 \mathrm{e} \sqrt{d}, \\
R_1 &=1 - \frac{1}{4}\sqrt{16-\delta^2}, & R_2 &=8 \sqrt{d}, & R_3 &= 9 \mathrm{e} \sqrt{d},
\end{align*}
and annuli $V_k(x) = B_{R_k}(x) \backslash \overline{B_{r_k}(x)} \subset \RR^{d+1}$ for $k \in \{ 1,2,3 \}$ and $x \in \RR^d$.
Note that for $k \in \{ 2,3 \}$, the set $V_k(z_j)$ will not necessarily be a subset of $\Gamma \times (-R_3, R_3)$ any more, but a subset of $\Gamma^+ \times (-R_3,R_3)$.
Exploiting again the relation $\Delta \Psi = \hat V \Psi$ almost everywhere on $\Gamma^+ \times (-T,T)$ from Lemma~\ref{lemma:properties_Psi} and Remark~\ref{rem:ext}, and using a Carleman estimate as in~\cite{NakicTTV-18}, we obtain
\begin{equation}
 \label{eq:after_Carleman_interpolation_2}
 \lVert
  \Psi
 \rVert_{H^1(V_2(z_j))}^2
 \leq
 g_1(\alpha)
 \lVert
  \Psi
 \rVert_{H^1(V_1(z_j))}^2
 +
 g_2(\alpha)
 \lVert
  \Psi
 \rVert_{H^1(V_3(z_j))}^2
\end{equation}
for all $j \in \Q$.
Here, $\alpha$ is an optimization parameter which can be chosen over some range $(\alpha_1, \infty)$, the prefactor $g_1(\alpha)$ is exponentially growing, and the prefactor $g_2(\alpha)$ is exponentially decreasing with respect to $\alpha$.
Furthermore, the dependence of $\alpha_1$ and the prefactors on $\delta$ and $\lVert V \rVert_\infty$ is explicit.
It should be noted that our choice of $R_2$ and $r_3$ differs from the one in~\cite{NakicTTV-18} but we still have the relation $\euler R_2 < r_3$ which ensured that $g_1$ was exponentially growing.
We sum Ineq.~\eqref{eq:after_Carleman_interpolation_2} over $j \in \Q$ and find
\begin{equation}
 \label{eq:after_Carleman_interpolation_2_sum}
 \sum_{j \in \Q}
 \lVert
  \Psi
 \rVert_{H^1(V_2(z_j))}^2
 \leq
 g_1(\alpha)
 \sum_{j \in \Q}
 \lVert
  \Psi
 \rVert_{H^1(V_1(z_j))}^2
 +
 g_2(\alpha)
 \sum_{j \in \Q}
 \lVert
  \Psi
 \rVert_{H^1(V_3(z_j))}^2.
\end{equation}
We now further estimate the three sums in Ineq.~\eqref{eq:after_Carleman_interpolation_2_sum}.
First, we note that since $R_1 < \delta$, the sets $V_1(z_j)$, $j \in \Q$, are mutually disjoint. Furthermore we have $V_1 (z_j) \cap \RR_+^{d+1} \subset S_1 (z_j)$ for $j \in \Q$, hence $\cup_{j \in \Q} V_1(z_j) \cap \RR^{d+1}_+ \subset U_{1,Z}$.
Combining this with the antisymmetry of $\Psi$ in the $(d+1)$-coordinate we find
\[
 \sum_{j \in \Q}
 \lVert
  \Psi
 \rVert_{H^1(V_1(z_j))}^2
 =
 \lVert
  \Psi
 \rVert_{H^1( \bigcup_{j \in \Q}  V_1(z_j) )}^2
 =
 2
 \lVert
  \Psi
 \rVert_{H^1( \bigcup_{j \in \Q}  V_1(z_j) \cap \RR^{d+1}_+)}^2
 \leq
 2 \lVert \Psi \rVert_{H^1(U_{1,Z})}^2 .
\]
Moreover, for every $(x,t) \in \Gamma^+ \times (-R_3,R_3)$, there are at most $(2 R_3 + 2)^d$ indices $j \in \Q$ such that $(x,t) \in V_3(z_j)$.
Thus, we have
\[
 \sum_{j \in \Q}
 \lVert
  \Psi
 \rVert_{H^1(V_3(z_j))}^2
 \leq
 (2 R_3 + 2)^d
 \lVert
  \Psi
 \rVert_{H^1(\Gamma^1 \times (-R_3,R_3))}^2.
\]
Finally, we claim that
\[
 \sum_{j \in \Q}
 \lVert
  \Psi
 \rVert_{H^1(V_2(z_j))}^2
 \geq
 \lVert
  \Psi
 \rVert_{H^1(X_1)}^2.
\]
This can be seen by distinguishing two cases. First we assume that $\Gamma$ is sufficiently large. More precisely we assume each elementary cell $(-1/2,1/2)^d + j \subset \Gamma$, $j \in \ZZ^d $,
has a next-to-one neighbor in $\Gamma$, that is, $(-1/2,1/2)^d + j + 2 f_k \subset \Gamma$ for some $k \in \{ 1, \dots, d \}$, where $f_k$ denotes the $k$-th unit vector or its negative. Then it is easy to see that $$\bigcup_{j \in \Q} V_2 (z_j) \supset X_1 .$$
\par
Let us now assume the contrary. Then $\diam \Gamma < 5\sqrt{d}$. Now $\cup_{j \in \Q} V_2 (z_j)$ does not necessarily cover $X_1$, but it covers a set $\tilde \Gamma \times (-1,1)$, where $\tilde \Gamma \subset \Gamma^+$ is a translate of $\Gamma$ and $\lVert \Phi \rVert_{H^1 (\tilde \Gamma \times (-1,1))} = \lVert \Phi \rVert_{H^1 (\Gamma \times (-1,1))}$.
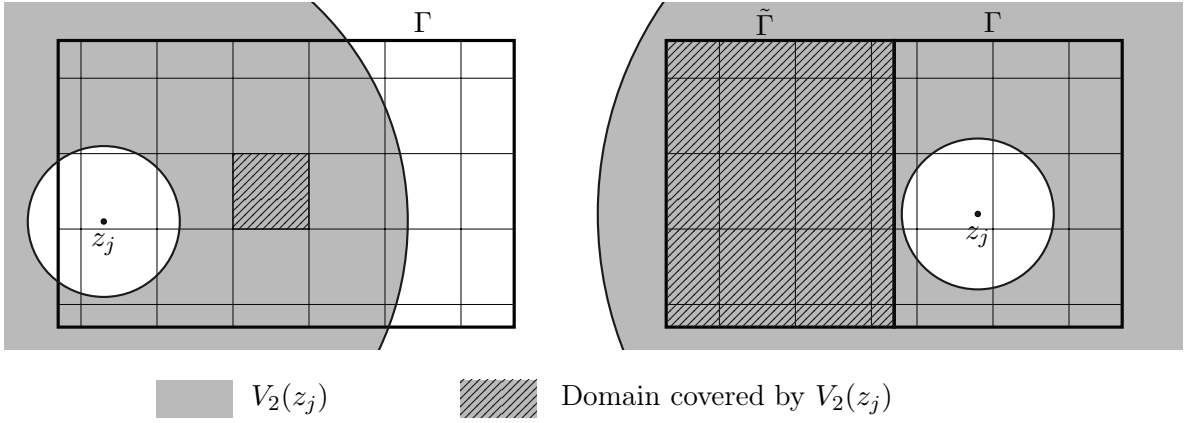
\begin{figure}[ht] \centering
  \begin{tikzpicture}
  \begin{scope}[xscale = -1]
   \begin{scope}
    \draw[very thin] (2.3,-.3) grid (5.3,3.5);
    \draw[very thick] (2.3,-.3) rectangle (5.3,3.5);
    \draw (4,3.75) node {$\Gamma$};
    \draw[very thin, shift={(-.4,0)}] (5.7,-.3) grid (8.7,3.5);
    \draw[very thick] (5.3,-.3) rectangle (8.3,3.5);
    \draw (7,3.75) node {$\tilde \Gamma$};
    \clip (1.5,-.6) rectangle (12,4);
    \draw[fill = annuli] (4.2,1.2) circle (1pt);
    \draw (4.2,.9) node {$z_j$};
    \filldraw[thick, annuli, even odd rule, fill opacity = 0.3] (4.2,1.2) circle (1cm) -- (4.2,1.2) circle (5cm);
    \filldraw[pattern=north east lines, pattern color = annuli] (5.3,-.3) rectangle (8.3,3.5);
   \end{scope}
   \begin{scope}[xshift = 10cm]
    \draw[very thin] (0.3,-.3) grid (6.3,3.5);
    \draw[very thick] (0.3,-.3) rectangle (6.3,3.5);
    \draw (1.5,3.75) node {$\Gamma$};
    \clip (0,-.6) rectangle (7,4);
    \draw[fill = annuli] (5.7,1.1) circle (1pt);
    \draw (5.7,.8) node {$z_j$};
   \filldraw[thick, annuli, even odd rule, fill opacity = 0.3]  (5.7,1.1) circle (1cm) -- (5.7,1.1) circle (4cm);
    \filldraw[pattern=north east lines, pattern color = annuli] (3,1) rectangle (4,2);
    \end{scope}
  \end{scope}
  \begin{scope}[yshift = -1.5cm, xshift = -15cm]
   \fill[annuli, even odd rule, fill opacity = 0.3] (0,0) rectangle (1,0.5);
   \draw (1.75,.25) node {$V_2(z_j)$};
   \begin{scope}[xshift = 4cm]
   \fill[annuli, even odd rule, fill opacity = 0.3] (0,0) rectangle (1,0.5);
   \fill[pattern=north east lines, even odd rule, pattern color = annuli] (0,0) rectangle (1,0.5);
   \draw (3.5,.25) node {Domain covered by $V_2(z_j)$};
   \end{scope}
  \end{scope}
\end{tikzpicture}
 \caption{The annuli $V_2(z_j)$ either each already cover another elementary cell in $\Gamma$, such that all of $\Gamma$ will be covered, or $\Gamma$ is so small that an annulus $V_2(z_j)$ already covers an entire translated copy $\tilde \Gamma$. (In this figure, the outer radii of $V_2(z_j)$ are not drawn according to its scale.)}
\end{figure}
\subsection{Proofs of Theorem~\ref{thm:result1} and Corollary~\ref{cor:scaling}}
\label{ssec:proofs}
To exploit the interpolation inequalities in the last section we need to show that the norms of the extended function $\Psi$ are comparable to norms of the original function $\psi$. This is possible, because $\Psi$ was chosen in an appropriate way in the first place.
\begin{proposition} \label{prop:upper_lower}
For all $1$-admissible $\Gamma \subset \RR^d$, $\Lambda \in \{\Gamma , \Gamma^+\}$, all $T > 0$, all $\tau \in (0,T]$, all real-valued $V\in L^\infty (\Gamma)$, all $E \in \RR$, and all $\psi \in \ran P_H (E)$ we have
\[
\frac{\tau}{2} \lVert \psi \rVert_{L^2 (\Lambda)}^2
\leq
\lVert \Psi \rVert_{H^1 (\Lambda \times (-\tau,\tau))}^2
\leq
2 \tau (1 + (1+\lVert V \rVert_\infty)\tau^2)
\euler^{2 \tau \sqrt{\max\{0,E\}} } \lVert \psi \rVert_{L^2 (\Lambda)}^2 .
\]
\end{proposition}
\begin{proof}
If $\Lambda = \Gamma^+$, we recall that $\Psi \colon \Gamma \times (-T,T) \to \CC$ is extended to $\Gamma^+ \times (-T,T)$ as explained in Remark~\ref{rem:ext}.
For $\tau > 0$ we have
\[
\lVert \Psi \rVert_{H^1 (\Lambda \times (-\tau,\tau))}^2
= \int_{-\tau}^\tau \int_{\Lambda} \left(
\lvert \partial_{t} \Psi \rvert^2 +
\lvert \nabla_d \Psi \rvert^2+
\lvert \Psi \rvert^2
\right) \drm x \, \drm t .
\]
Using Green's theorem we have by Lemma~\ref{lemma:properties_Psi} and Remark~\ref{rem:ext}
\[
  \int_{\Lambda} \lvert \nabla_d \Psi \rvert^2 \drm x
  =
  -\int_{\Lambda} (\Delta_d \Psi) \overline{\Psi} \drm x
  =
  - \int_{\Lambda} V\lvert \Psi\rvert^2 \drm x +
 \int_{\Lambda} (\partial^2_{t} \Psi) \overline{\Psi} \drm x
\]
for all $t \in \RR$.
First, we estimate
\begin{align*}
 \lVert \Psi \rVert_{H^1 (\Lambda \times (-\tau,\tau))}^2
 &=
 \int_{-\tau}^\tau \int_{\Lambda} \left(
 \lvert \partial_{t} \Psi \rvert^2  - V\lvert \Psi\rvert^2 + (\partial^2_{t} \Psi) \overline{\Psi}
 + \lvert \Psi \rvert^2 \right) \drm x \, \drm t \\
 &\leq
 \int_{-\tau}^\tau \int_{\Lambda} \left(
\lvert \partial_{t} \Psi \rvert^2  + (\partial^2_{t} \Psi) \overline{\Psi}
+ (1+ \lVert V \rVert_\infty) \lvert \Psi \rvert^2 \right) \drm x \, \drm t \\
&=  2 \int_{-\infty}^E I (\lambda) \drm \lVert P_H(\lambda) \psi \rVert^2 ,
\end{align*}
where
\begin{align*}
  I (\lambda)
 &= \int_{0}^\tau \left( (1+\lVert V \rVert_\infty)\funs_t (\lambda)^2 + (\partial_t \funs_t (\lambda))^2 + (\partial_t^2 \funs_t (\lambda) )\funs_t (\lambda) \right) \drm t \\
 &= (1+\lVert V \rVert_\infty) \int_0^\tau  \funs_t(\lambda)^2 \drm t + (\partial_t \funs_t (\lambda))\funs_\tau (\lambda) .
\end{align*}
In order to estimate $I(\lambda)$ for $\lambda \in (-\infty , E]$ from above, we distinguish two cases.
If $\lambda \leq 0$ we use $s_t (\lambda) \leq t$ and $(\partial_t s_t (\lambda))  s_\lambda(t) \leq t$ for $t > 0$, and obtain
\[
 I (\lambda) \leq (1 + \lVert V \rVert_\infty) \tau^3/3 + \tau .
\]
If $E,\lambda > 0$ we use $\sinh(\sqrt{\lambda} t)/\sqrt{\lambda} \leq t \cosh(\sqrt{\lambda} t)$ for $t > 0$, and $\cosh^2(\sqrt{\lambda} \tau) \leq \euler^{2 \sqrt{\lambda} \tau}$, and obtain for all $\lambda \in (0 , E]$
\begin{align*}
 I (\lambda) &= (1+\lVert V \rVert_\infty) \int_0^\tau \frac{\sinh^2(\sqrt{\lambda} t)}{\lambda} \drm t + \sinh(\sqrt{\lambda} \tau) \cosh(\sqrt{\lambda} \tau)/\sqrt{\lambda} \\
&\leq ((1 + \lVert V \rVert_\infty ) T^3 \cosh^2(\sqrt{\lambda} \tau) + \tau \cosh^2(\sqrt{\lambda} \tau) ) \leq ((1 + \lVert V \rVert_\infty ) \tau^3 + \tau) \euler^{2 \sqrt{E} \tau}.
\end{align*}
Hence, we conclude the upper bound of the statement. For the lower bound we drop the gradient term and obtain
 \begin{align*}
\lVert \Psi \rVert_{H^1 (\Lambda \times (-\tau,\tau))}^2 &\ge
\int_{-\tau}^\tau \int_{\Lambda} \left(
\lvert \partial_{t} \Psi \rvert^2 +
\lvert \Psi \rvert^2
\right) \drm x \, \drm t =
2 \int_{-\infty}^E \tilde I (\lambda) \drm \lVert P_H(\lambda) \psi \rVert^2 ,
\end{align*}
where
\[
 \tilde I (\lambda) = \int_{0}^\tau\left[ \funs_t (\lambda)^2+ (\partial_t \funs_t (\lambda))^2 \right] \drm t .
\]
If $\lambda = 0$, the lower bound $\tilde I (\lambda) \geq \tau$ follows immediately.
Else, we use
$(\partial_t \funs_t (\lambda))^2 \geq \cos(\sqrt{\lvert \lambda \rvert} t)$ to obtain
\[
 \tilde I (\lambda) \geq
 \int_0^\tau \cos^2(\sqrt{\lvert \lambda \rvert} t) \drm t
 =
 \frac{\tau}{2} + \frac{\sin(2 \sqrt{\lvert \lambda \rvert} \tau)}{4 \sqrt{\lvert \lambda \rvert}}.
\]
Now, if $2 \sqrt{\lvert \lambda \rvert} \tau < \pi$, the sinus term is positive and we drop it to find $ \tilde I (\lambda) \geq
 \tau/2$.
If $2 \sqrt{\lvert \lambda \rvert} \tau \geq \pi$, we have $\sin(2 \sqrt{\lvert \lambda \rvert} \tau ) \geq -1$ and estimate
\[
 \tilde I (\lambda)
 \geq
 \frac{\tau}{2} - \frac{1}{4 \sqrt{\lvert \lambda \rvert}}
 =
 \frac{\tau}{2} - \frac{\pi}{4 \pi\sqrt{\lvert \lambda \rvert}}
 \geq
 \frac{\tau}{2} - \frac{\tau}{2 \pi}
 \geq
 \frac{\tau}{4}.
\]
Hence, we conclude the lower bound of the statement.
\end{proof}
Now we are in the position to complete the proofs of Theorem~\ref{thm:result1} and Corollary~\ref{cor:scaling}.

\begin{proof}[Proof of Theorem~\ref{thm:result1}]
Since $\ran P_H(E) \subset \ran P_H(0)$ for $E < 0$ it suffices to prove the statement in the case $E \geq 0$.
We use the upper bound of Proposition~\ref{prop:upper_lower} with $T = R_3$, the lower bound with $T = 1$, and $\lVert \psi \rVert_{L^2 (\Gamma^+)} \leq (3 R_3)^d \lVert \psi \rVert_{L^2 (\Gamma)}$
to obtain
\begin{align*}
  \frac{\lVert \Psi \rVert_{H^1 (\Gamma^+ \times (-R_3,R_3))}^2 }{\lVert \Psi \rVert_{H^1 (X_1)}^2}
  \leq
  4 \cdot R_3 (1 + (1+\lVert V \rVert_\infty) R_3^2) \euler^{2 R_3 \sqrt{E}} (3R_3)^d =:  D_3^2 .
\end{align*}
Then Propositions~\ref{prop:interpolation1} and \ref{prop:interpolation2} imply
\begin{align*}
\lVert \Psi \rVert_{H^1 (\Gamma^+ \times (-R_3,R_3))} &\leq D_3 \lVert \Psi \rVert_{H^1 (X_1)} \\
&\leq D_1^\gamma D_2 D_3
\lVert \Psi \rVert_{H^1 (\Gamma^+ \times (-R_3,R_3))}^{1 - \gamma}
\lVert (\partial_{t} \Psi) (\cdot , 0) \rVert_{L^2 (\equiset)}^{\gamma / 2}
\lVert \Psi \rVert_{H^1 (U_{3,Z})}^{\gamma / 2} .
\end{align*}
We recall $R_3 = 9 \euler \sqrt{d}$. By the definition of $U_{3,Z}$ and $\Gamma^+$ we have $U_{3,Z} \subset \Gamma^+ \times (-R_3,R_3)$. Hence,
\[
\lVert \Psi \rVert_{H^1 (\Gamma^+ \times (-R_3,R_3))} \leq
 D_1^2 D_2^{2/\gamma} D_3^{2/\gamma}
\lVert (\partial_{t} \Psi) (\cdot , 0) \rVert_{L^2 (\equiset)} .
\]
By Proposition~\ref{prop:upper_lower}, the square of the left hand side is bounded from below by
\[
\lVert \Psi \rVert_{H^1 (\Gamma^+ \times (-R_3,R_3))}^2
\geq
\frac{R_3}{2} \lVert \psi \rVert_{L^2 (\Gamma^+)}^2.
\]
 Putting everything together we obtain by using $(\partial_{t} \Psi) (\cdot , 0) = \psi$
 \begin{equation*}
 \frac{R_3}{2} \lVert \psi \rVert_{L^2 (\Gamma)}^2
  \leq
  \frac{R_3}{2} \lVert \psi \rVert_{L^2 (\Gamma^+)}^2 \leq
 D_1^4 \left( D_2 D_3 \right)^{4/\gamma} \lVert \psi \rVert_{L^2 (\equiset )}^2 .
\end{equation*}
>From the definitions of $D_k$, $k \in \{1,2,3\}$, and $\gamma$ one calculates that
\[
  D_1^4 \left( D_2 D_3 \right)^{4/\gamma} 2 R_3^{-1} \leq C_\sfuc^{-1}
  =
  \delta^{-N \bigl(1 + \lVert V \rVert_\infty^{2/3} + \sqrt{E} \bigr)}
\]
with some constant $N > 0$ depending only on the dimension.
%
\end{proof}
\begin{proof}[Proof of Corollary~\ref{cor:scaling}]
First note that
\[
\psi \in \ran P_H (E) \Leftrightarrow \psi \in \ran P_{H-\lambda} (E-\lambda)
\]
thus integrals of $\psi $ can be estimated by using either of the two subspace properties.
Now it is possible to perform an optimization over the spectral shift parameter $\lambda$.

To analyze the dependence of estimate on the scale parameter $G$, fix $V \in L^\infty (\Gamma)$ and $\psi \in \ran P_H (E)$. We define $\tilde V \colon G^{-1} \Gamma \to \RR$ by $\tilde V (x) =  G^2 V (Gx)$, the operator $\tilde H = -\Delta + \tilde V$ on $L^2 (G^{-1} \Gamma)$, the bounded operator $S \colon L^2(\Gamma)\to L^2(G^{-1} \Gamma)$ by $(S f)(x) = f(Gx)$, and $\tilde \psi \in L^2 (G^{-1} \Gamma)$ by $\tilde \psi = S \psi$. By a straightforward calculation one shows $\tilde H = G^2 S H  S^{-1}$.
We also define the mapping $\RR \ni \lambda \mapsto \hat P(\lambda) = S P_H(\lambda) S^{-1}$.
Then $\{\hat P(\lambda) \colon \lambda \in \RR \}$ is the resolution of identity corresponding to the operator $S H S^{-1}$.
This follows from the formula $S^{\ast} = G^{-d} S^{-1}$.
By assumption we have $P_H (E) \psi = \psi$, hence $\hat P (E) \tilde \psi = \tilde \psi$. Since  $P_{\tilde H}(\lambda) = \hat P (G^{-2} \lambda)$ this implies that
\begin{equation*}
  \tilde \psi \in \ran P_{\tilde H} (G^2 E) .
\end{equation*}
By construction we have
\[
G^{-1} \Gamma = \bigtimes_{i=1}^d (G^{-1} \alpha_i , G^{-1} \beta_i),
\]
$G^{-1} \beta_i - G^{-1} \alpha_i \geq 1$, $G^{-1} \equiset$ arises from a $(1,\tilde\delta)$-equidistributed sequence where $\tilde\delta = \delta / G \in (0,1/2)$. Hence we can apply Theorem~\ref{thm:result1} and obtain by substitution
 \[
  \lVert \psi \rVert_{L^2 (\equiset)}^2
  =
  G^d \lVert \tilde \psi \rVert_{L^2 (G^{-1}\equiset)}^2
  \geq
  C_\sfuc^G G^d\lVert \tilde\psi \rVert_{L^2 (G^{-1} \Gamma)}^2
  =
  C_\sfuc^G \lVert \psi \rVert_{L^2 (\Gamma)}^2 ,
 \]
 where $C_\sfuc^G = C_\sfuc (d,\delta / G , E G^2 , \lVert V \rVert_\infty G^2)$ with $C_\sfuc$ from Theorem~\ref{thm:result1}.
 \end{proof}
\section{Perturbation of spectral edges and eigenvalues in gaps}
\label{sec:perturb}

Throughout this section all occurring operators are defined on the same infinite dimensional Hilbert space $\cH$
and we denote the domain of an operator $A$ by $\cD(A)$.
In this section we prove lifting estimates for edges of the components of the essential spectrum and also for the eigenvalues in spectral gaps.
As a starting point we recall that the notion of `spectral edges' is well defined and gives rise to a Lipschitz-continuous function $t \mapsto f(t)$ of a coupling constant.

\begin{lemma} \label{lem:lipschitz_left_edge}
Let $A$ be self-adjoint, $B\neq 0$ bounded, symmetric and non--negative.
Let $a\in \sigma_{\mathrm{ess}}(A)$ and let $b > a$ be such that $(a,b)\cap \sigma_{\mathrm{ess}}(A) = \emptyset$. Let
\[
  t_0 = (b-a)/\lVert B \rVert
  \quad\text{ and }\quad
  t_- = \max \{0, - t\}.
\]
Then
\[
  f\colon (-t_0,t_0) \to \mathbb{R},\quad
  f(t) = \sup \left( \sigma_{\mathrm{ess}}(A + t B) \cap (-\infty, b-t_- \lVert B \rVert) \right),
\]
satisfies for all $t\in (-t_0,t_0)$,
\begin{enumerate}[(a)]
  \item $f(0) = a$,
  \item $f(t)\in \sigma_{\mathrm{ess}}(A + t B)$,
  \item $(f(t), b - t_- \lVert B \rVert) \cap \sigma_{\mathrm{ess}}(A + t B) = \emptyset $, and
  \item $f$ is Lipschitz continuous with Lipschitz coefficient $\lVert B \rVert $.
\end{enumerate}
\end{lemma}

\begin{remark}
There is an analogous version of Lemma~\ref{lem:lipschitz_left_edge} for the case where $(a,b)\cap \sigma_{\mathrm{ess}}(A) = \emptyset$ and $b \in \sigma_\ess(A)$,
cf.~Cor.~\ref{cor:lifting_lower_edge}.
Lemma~\ref{lem:lipschitz_left_edge} holds as well in the case of indefinite $B$, provided one replaces $t_0$ by $ (b-a)/(2 \lVert B \rVert) $
and claim (c) by   $(f(t), b - \vert t\vert \lVert B \rVert) \cap \sigma_{\mathrm{ess}}(A +  t B) = \emptyset $.
Obviously, an analog of Lemma~\ref{lem:lipschitz_left_edge} holds also for isolated eigenvalues.
See for instance \cite{Veselic-K-08} and \cite{Seelmann-14-diss} for some background.
\end{remark}

Lemma~\ref{lem:lipschitz_left_edge}  follows immediately from

\begin{lemma}
 \label{lem:lemma_albrecht_s}
 Let $A$ be self-adjoint and $B$ bounded and non-negative.
 If $(a,b) \subset \RR$, then
 \begin{equation}
 \label{eq:seelmann_1}
  (a,b) \cap \sigma_\ess(A) = \emptyset
  \quad
  \Rightarrow
  \quad
  (a + \lVert B \rVert, b) \cap \sigma_\ess(A+B) = \emptyset,
 \end{equation}
 and
 \begin{equation}
 \label{eq:seelmann_2}
  (a,b) \cap \sigma_\ess(A) \neq \emptyset
  \quad
  \Rightarrow
  \quad
  (a, b + \lVert B \rVert) \cap \sigma_\ess(A + B) \neq \emptyset.
 \end{equation}
 Here we use the convention that $(c,d) = \emptyset$ if $c \geq d$.
\end{lemma}

\begin{remark}
 \label{rem:albrecht_s}
 Note that since the essential spectrum is closed, the statements of Lemma~\ref{lem:lemma_albrecht_s}
also hold if the corresponding intervals are replaced by closed or semi-closed intervals.
If one is interested in negative perturbations $-B$, one applies the contraposition of  \eqref{eq:seelmann_1} and \eqref{eq:seelmann_2}.
\end{remark}

\begin{proof}[Proof of Lemma~\ref{lem:lemma_albrecht_s}]
Let us assume $(a,b) \cap \sigma_\ess(A) = \emptyset$.
Then for every $\epsilon > 0$, $A$ has at most finitely many eigenvalues with finite multiplicity in $(a + \epsilon, b - \epsilon)$.
Thus, there is a finite rank perturbation $T$ such that $(a + \epsilon, b - \epsilon) \cap \sigma (A + T) = \emptyset$.
 From Proposition~2.1 in \cite{Seelmann-17-arxiv} we infer that $(a + \epsilon + \lVert B \rVert, b - \epsilon) \cap \sigma(A + T + B) = \emptyset$.
Since finite rank perturbations leave the essential spectrum unchanged, we obtain
 \[
  (a + \epsilon + \lVert B \rVert , b - \epsilon ) \cap \sigma_\ess(A + B) = \emptyset
  \quad
  \text{for all}
  \quad
  \epsilon > 0.
 \]
 This shows~\eqref{eq:seelmann_1}.
 The relation~\eqref{eq:seelmann_2} is equivalent to \eqref{eq:seelmann_1} by contraposition.
\end{proof}

Lemma~\ref{lem:lipschitz_left_edge} provides an \emph{upper bound} on the movement speed of spectral edges.
In this section, we complement this by providing \emph{lower bounds} on the lifting of spectral edges and eigenvalues in gaps of the essential spectrum.
All the results hold under an abstract positivity condition on the perturbation.
This condition is in particular satisfied for Sch\"odinger operators with appropriate perturbations as in Section~\ref{ssec:model}.
This connection will be spelled out in Section~\ref{sec:application}.

\subsection{Below the essential spectrum}
As a warm up and to introduce notation we first consider the infimum of the essential spectrum and eigenvalues below it.
For a lower semi-bounded self-adjoint operator $A$, we set $\lambda_\infty (A) = \inf \sigma_{\ess} (A)$. Note that $\lambda_\infty (A) = \infty$ if $\sigma_\ess (A) = \emptyset$.
Moreover, we denote by $\{ \lambda_{k} (A) \}_{k \in \NN}$ the sequence of eigenvalues of $A$ below $\sigma_{\mathrm{ess}} (A)$, enumerated non-decreasingly and counting multiplicities.
If there are exactly $N \in \NN_0$ eigenvalues in $(-\infty , \lambda_\infty (A))$ then we set $\lambda_{k} (A) =  \lambda_\infty (A)$ for all $k \in \NN$ with $k > N$.
If $A$ has purely discrete spectrum then the sequence $\lambda_k(A)$, $k \in \NN$, is the non-decreasing sequence of eigenvalues of $A$.

\begin{lemma}
 \label{lem:bottom}
 Let $A$ be self-adjoint and lower semi-bounded, $B$ bounded and symmetric, $E \in \RR$, $\kappa \in \RR$, and
 \[
 \forall x \in \ran (P_{A + B} (E)) \colon \quad
 \langle x , B x \rangle \geq \kappa  \lVert x \rVert^2  .
 \]
 Then for all $k \in \NN \cup \{ \infty \}$ such that $\lambda_k(A + B) < E$, we have
 \[
  \lambda_k(A + B) \geq \lambda_k(A) + \kappa.
 \]
\end{lemma}

\begin{proof}
Let $\epsilon_0 = E - \lambda_k(A + B) > 0$. Then we have by assumption
 \begin{align*}
  \lambda_k(A + B)
  &=
  \inf_{0 < \epsilon \leq \epsilon_0}
  \sup_{\substack{x \in \ran P_{A + B} (\lambda_k (A + B) + \epsilon)\\ \lVert x \rVert = 1}}
  \left( \langle x, A x \rangle + \langle x, B x \rangle \right) \\[1ex]
  &\geq
  \inf_{0 < \epsilon \leq \epsilon_0}
  \sup_{\substack{x \in \ran P_{A + B} (\lambda_k (A + B) + \epsilon)\\ \lVert x \rVert = 1}}
  \langle x, A x \rangle + \kappa .
 \end{align*}
 Since $\dim \ran P_{A + B}( \lambda_k(A + B) + \epsilon) \geq k$ for $\epsilon > 0$, we have by the standard variational principle
  \begin{align*}
  \sup_{\substack{x \in \ran P_{A + B} (\lambda_k (A + B) + \epsilon)\\ \lVert x \rVert = 1}}
  \langle x, A x \rangle
  &=
  \sup_{\substack{\cL \subset
  \ran P_{A + B} (\lambda_k (A + B) + \epsilon)\\ \dim \cL= k}}
  \sup_{\substack{x \in \cL\\ \lVert x \rVert = 1}}
  \langle x, A x \rangle
   \\
  &\geq
  \inf_{\substack{\cL \subset
  \ran P_{A + B} (\lambda_k (A + B) + \epsilon)\\ \dim \cL= k}}
  \sup_{\substack{x \in \cL\\ \lVert x \rVert = 1}}
  \langle x, A x \rangle\\
  &\geq
  \inf_{\substack{\cL \subset
  \cD(A)\\ \dim \cL= k}}
  \sup_{\substack{x \in \cL\\ \lVert x \rVert = 1}}
  \langle x, A x \rangle
  = \lambda_k(A).
  \qedhere
  \end{align*}
\end{proof}

\subsection{Ordering from right to left in gaps}
\label{ssec:rightleft}

For a self-adjoint (not necessarily semi-bounded) operator $A$, and $\gamma \in \rho(A) \cap \RR$, we set $\lambda_{\infty , \gamma}^\leftarrow(A) = \sup \{ \lambda < \gamma \colon \lambda \in \sigma_\mathrm{ess}(A) \}$.
If there is no essential spectrum below $\gamma$ this means $\lambda_{\infty , \gamma}^\leftarrow(A)  = -\infty$,
otherwise $\lambda_{\infty , \gamma}^\leftarrow(A) $ is the right edge of the component of $\sigma_\mathrm{ess}(A)$ below $\gamma$.
Moreover, we denote by $\{ \lambda_{k , \gamma}^\leftarrow(A) \}_{k \in \NN}$ the sequence of eigenvalues of $A$ in $( \lambda_{\infty , \gamma}^\leftarrow(A), \gamma)$,
enumerated non-increasingly and counting multiplicities.
If there are infinitely many eigenvalues inside $(\lambda_{\infty , \gamma}^\leftarrow(A) , \gamma)$, then it follows that
$\lambda_{\infty , \gamma}^\leftarrow(A)= \inf_k \lambda_{k , \gamma}^\leftarrow(A)= \lim_k \lambda_{k , \gamma}^\leftarrow(A)   $.
If there are exactly $N \in \NN_0$ eigenvalues inside $(\lambda_{\infty , \gamma}^\leftarrow(A) , \gamma)$ we set $\lambda_{k,\gamma}^\leftarrow (A) = \lambda_{\infty , \gamma}^\leftarrow(A)$ for all $k \in \NN$ with $k > N$.
Moreover, we define
\[
 \mathcal{M}_\gamma^- = \{ \mathcal{M} \colon \mathcal{M} \ \text{is a maximal} \ (A-\gamma)\text{-non-positive subspace of} \ \cD (A) \} .
\]
Recall that a subspace $\cM \subset \cH$ is called \emph{$A$-non-positive} if $\langle x , Ax  \rangle \leq 0$ for all $x \in \cM$.
Such a $\cM$ is called \emph{maximal} if there is no $A$-non-positive subspace which properly contains $\cM$.
\par
For $\gamma \in \rho (A) \cap \RR$ we define
\begin{align*}
\dist^{\leftarrow} (\gamma , \sigma (A)) &= \dist (\gamma , \sigma (A) \cap (-\infty , \gamma]) =\gamma -\sup \left( \sigma(A) \cap (-\infty,\gamma]\right), \ \text{and} \\
\dist^{\rightarrow} (\gamma , \sigma (A)) &= \dist (\gamma , \sigma (A) \cap [\gamma , \infty)) = \inf (\sigma (A) \cap [\gamma , \infty)) - \gamma .
\end{align*}
Note that $\dist^{\leftarrow} (\gamma , \sigma (A)) = \infty$ if $\gamma < \inf \sigma (A)$, and that $\dist^{\rightarrow} (\gamma , \sigma (A)) = \infty$ if $\gamma > \sup \sigma (A)$.
\begin{lemma}
\label{lem:maximal_non-positive}
Let $A$ be self-adjoint, $\gamma \in \rho (A) \cap \RR$, $B$ bounded and symmetric, satisfying
\begin{enumerate}[(i)]
 \item
 $\lVert B \rVert < \frac{1}{2} \dist(\gamma, \sigma(A))$, or
 \item
 $0 \leq B$ and $\lVert B \rVert < \dist(\gamma, \sigma(A))$.
\end{enumerate}
Then $\ran P_{A + B}(\gamma) \cap \cD(A)$ is a maximal $(A - \gamma)$-non-positive subspace of $\cD(A)$.
\end{lemma}

\begin{proof}
 For all $x \in \ran P_{A+B} (\gamma) \cap \cD(A)$ with $\lVert x \rVert = 1$ we have
 \[
   \langle x , (A - \gamma ) x\rangle
   =
   \langle x , (A+B - \gamma ) x \rangle
   -
   \langle x , B x \rangle
   \leq
   - \dist^\leftarrow(\gamma, \sigma(A + B))
   - \langle x , B x \rangle.
 \]
 If (ii) is satisfied, this is clearly negative.
 If (i) is satisfied, we use
 $\dist^\leftarrow(\gamma, \sigma(A+B))
 \geq
 \dist(\gamma, \sigma(A)) - \lVert B \rVert$
 and conclude
 \[
  \langle x , (A - \gamma ) x\rangle
  \leq
  - \dist^\leftarrow(\gamma, \sigma(A)) + \lVert B \rVert - \langle x , B x \rangle
  \leq
  - \dist^\leftarrow(\gamma, \sigma(A)) + 2 \lVert B \rVert
  < 0.
 \]
 Hence, $\ran P_{A + B}(\gamma) \cap \cD(A)$ is an $(A - \gamma)$-non-positive subspace of $\cD(A)$.

Let us assume it is not maximal.
Then we can choose $x \in \ran P_{A+B} (\gamma)^\perp \cap \cD (A)$ satisfying $\lVert x \rVert = 1$ and $\langle x , (A - \gamma) x \rangle \leq 0$.
In case (i), we use $\lVert B \rVert < \dist(\gamma, \sigma(A))/2$, to obtain
\[
 \langle x , (A+B - \gamma ) x \rangle
 \geq \dist (\gamma , \sigma (A+B))
 \geq \dist (\gamma , \sigma (A)) - \lVert B \rVert
 \geq \frac{1}{2}\dist (\gamma , \sigma (A)) .
\]
This leads to the contradiction
\[
 0 \geq \langle x , (A-\gamma) x \rangle = \langle x , (A+B-\gamma) x \rangle - \langle x , B x \rangle \geq \frac{1}{2}\dist (\gamma , \sigma (A)) - \lVert B \rVert > 0 .
\]
In case (ii), we use
$
 \langle x , (A+B - \gamma ) x \rangle
 \geq \dist^\rightarrow (\gamma , \sigma (A)) ,
$ and find the contradiction
\[
 0 \geq \langle x , (A-\gamma) x \rangle = \langle x , (A+B-\gamma) x \rangle - \langle x , B x \rangle \geq \dist^\rightarrow (\gamma , \sigma (A)) - \lVert B \rVert > 0 .
\qedhere
\]
\end{proof}

The following theorem is a reformulation of Theorem~3.1 in \cite{LangerS-16}, obtained by replacing $T$ by $-T$, and by working with operator domains instead of quadratic form domains.
\begin{theorem} \label{thm:LangerS}
Let $A$ be self-adjoint.
For all $\gamma \in \rho (A) \cap \RR$ and $k \in \NN$ we have
\begin{equation}\label{eq:LS16}
 \lambda_{k , \gamma}^\leftarrow (a)
 =
 \inf_{\mathcal{M} \in \mathcal{M}_\gamma^-} \inf_{\substack{\cL \subset \mathcal{M} \\ \dim \mathcal L = k-1}} \sup_{\substack{x \in \mathcal{M} \\ x \perp \cL \\ \lVert x \rVert = 1}} \langle x , A x \rangle .
\end{equation}
\end{theorem}
\begin{remark}
 If  $\gamma < \inf \sigma(A)$, then by definition we have $\lambda_{k , \gamma}^\leftarrow (A) = -\infty$. Since we also have $\mathcal{M}_\gamma^- = \{0\}$, the last supremum on the right hand side of Eq.~\eqref{eq:LS16} is taken over the empty set. Hence, the right hand side of Eq.~\eqref{eq:LS16} is as well minus infinity.
 \par
 In our theorems in Section~\ref{ssec:rightleft}, \ref{ssec:leftright} and \ref{sec:application}, we will be faced with such a situation from time to time, in particular, whenever we consider spectrum to the left of $\gamma$ with $\gamma < \inf \sigma (A)$, or spectrum to the right of $\gamma$ with $\gamma > \sup \sigma (A)$.
\end{remark}
With Lemma~\ref{lem:maximal_non-positive}  and the last theorem at disposal we are prepared to prove
\begin{theorem} \label{thm:gaps_left}
Let $A$ be self-adjoint, $\gamma \in \rho (A) \cap \RR$, $\kappa \in \RR$, $B$ bounded and symmetric, satisfying
\begin{enumerate}[(i)]
 \item
 $\lVert B \rVert < \frac{1}{2} \dist(\gamma, \sigma(A))$, or
 \item
 $0 \leq B$ and $\lVert B \rVert < \dist(\gamma, \sigma(A))$.
\end{enumerate}
Assume further
\[
 \forall x \in \ran (P_{A+B} (\gamma)) \colon \quad
 \langle x , B x \rangle \geq \kappa \lVert x \rVert^2 .
\]
Then for all $k \in \NN \cup \{\infty\}$, we have
 \[
\lambda_{k , \gamma}^\leftarrow  (A + B)
\geq
 \lambda_{k , \gamma}^\leftarrow  (A) + \kappa .
\]
\end{theorem}
\begin{proof}
First consider $k \in \NN$.
Since $\lVert B \rVert < \dist (\gamma , \sigma (A) )$ we have $\gamma \in \rho (A + B)$.
We apply the standard variational principle to $-(A + B)|_{\ran P_{A+B} (\gamma)}$ and obtain
\begin{align*}
 \lambda_{k , \gamma}^\leftarrow (A + B) &=
 \inf_{\substack{\cL \subset \ran P_{A+B} (\gamma) \cap \cD(A) \\ \dim \cL = k-1}}
 \sup_{\substack{x \in \ran P_{A+B} (\gamma) \cap \cD(A)\\ x \perp \cL \\ \lVert x \rVert = 1}} \langle x , (A+B) x \rangle  \\
  &\geq \inf_{\substack{\cL \subset \ran P_{A+B} (\gamma) \cap \cD(A) \\ \dim \cL = k-1}} \sup_{\substack{x \in \ran P_{A+B} (\gamma) \cap \cD(A) \\ x \perp \cL \\ \lVert x \rVert = 1}}
\langle x , A x \rangle + \kappa .
\end{align*}

By Lemma~\ref{lem:maximal_non-positive}, the subspace $\ran P_{A+B} (\gamma) \cap \cD(A)$ is a maximal $(A-\gamma)$-non-positive subspace of $\cD (A)$.
Hence,
\[
\lambda_{k , \gamma}^\leftarrow (A + B)
\geq
\inf_{\mathcal{M} \in \mathcal{M}_\gamma^-} \inf_{\substack{\cL \subset \mathcal{M} \\ \dim \cL = k-1}} \sup_{\substack{x \in \mathcal{M} \\ x \perp \cL \\ \lVert x \rVert = 1}}
\langle x , A x \rangle + \kappa .
\]
For $k \in \NN$, the statement of the theorem follows from Theorem~\ref{thm:LangerS}.
\par
If there are infinitely many eigenvalues $ \lambda_{k , \gamma}^\leftarrow (A+B)$ in $( \lambda_{\infty , \gamma}^\leftarrow (A+B),\gamma)$ then
\[
\lambda_{\infty , \gamma}^\leftarrow (A+B)
= \inf_{k \in \NN} \lambda_{k , \gamma}^\leftarrow (A+B)
\geq
 \inf_{k \in \NN} \lambda_{k , \gamma}^\leftarrow (A) +\kappa
=
\lambda_{\infty , \gamma}^\leftarrow (A)+\kappa. \qedhere
\]
\end{proof}
We can also require a lower bound on $B$ in terms of another operator, rather than a scalar.
Slight modifications of the last proof yield the following
\begin{corollary}\label{thm:gaps_left-gen}
Let  $A$ be self-adjoint, $B,C$ bounded and symmetric, satisfying $\gamma \in \rho (A+C) \cap \RR$,
 $\langle x , B x \rangle \geq  \langle x , C x \rangle$ for all $x \in \ran (P_{A+B} (\gamma)) $
and
\begin{enumerate}[(i)]
 \item
 $\lVert B-C \rVert < \frac{1}{2} \dist(\gamma, \sigma(A+C))$, or
 \item
 $0 \leq B-C$ and $\lVert B-C \rVert < \dist(\gamma, \sigma(A+C))$.
\end{enumerate}
Then for all $k \in \NN \cup \{\infty\}$, we have
 \[
\lambda_{k , \gamma}^\leftarrow  (A + B)
\geq
 \lambda_{k , \gamma}^\leftarrow  (A+C) .	
\]
\end{corollary}

In the case of a non-negative perturbation, one might ask whether 
lifting estimates as in Theorem \ref{thm:gaps_left}
are valid as long as the norm of $B\geq0$ is smaller than the distance
between the reference point $\gamma \in \rho(A) \cap \RR$ and the closest spectral value below $\gamma$. The following theorem gives a positive answer to this question, however, under a stronger assumption on the positivity of $B$.
Combining Lemma~\ref{lem:maximal_non-positive} (ii) and Theorem~\ref{thm:LangerS} we are able to prove
\begin{theorem} \label{thm:gaps_left_opt}
Let $A$ be self-adjoint, $\gamma \in \rho (A) \cap \RR$,
$\kappa \in \RR$, $B$ bounded and symmetric satisfying $0 \leq B$ and $\lVert B \rVert < \dist^{\leftarrow} (\gamma , \sigma (A))$.
Let $n$ be the smallest integer larger than $\lVert B \rVert / \dist^\rightarrow (\gamma , \sigma (A))$,
and assume that
\[
 \forall x \in \bigcup_{j=1}^n\ran (P_{A+ j B/n} (\gamma)) \colon \quad
 \langle x , B x \rangle \geq \kappa  \lVert x \rVert^2  .
\]
Then for all $k \in \NN \cup \{\infty\}$, we have
 \[
\lambda_{k , \gamma}^\leftarrow  (A + B)
\geq
 \lambda_{k , \gamma}^\leftarrow  (A) + \kappa .
\]
\end{theorem}
\begin{proof}
First assume additionally that $\lVert B \rVert < \dist^{\rightarrow} (\gamma , \sigma (A))$. Then we have $0 \leq B$ and $\lVert B \rVert < \dist (\gamma , \sigma (A))$, and the statement follows from Theorem~\ref{thm:gaps_left}.
\par
Now we drop the assumption $\lVert B \rVert < \dist^{\rightarrow} (\gamma , \sigma (A))$, and consider the case $k \in \NN \cup \{\infty\}$, $0 \leq B$, and $\lVert B \rVert < \dist^{\leftarrow} (\gamma , \sigma (A))$.
Recall that $n \in \NN$ satisfies $0 \leq B / n$ and $\lVert B / n \rVert < \dist^{\rightarrow} (\gamma , \sigma (A))$.
It follows that
\[
 0 \leq B/n \ \text{and} \ \lVert B/n \rVert < \dist(\gamma, \sigma(A + j B/n))
 \quad
 \text{for}\
 j \in \{0, \ldots, n-1\},
\]
see~\cite[Proposition~2.1]{Seelmann-17-arxiv}.
We now apply the result obtained above iteratively for $j \in \{0,\ldots, n-1\}$ to the operator $A + j B / n$ instead of $A$, with the perturbation $B$ replaced by $B / n$, and with $\kappa$ replaced by $\kappa/n$.
\end{proof}
If we are interested in a lifting estimate for edges of $\sigma_{\ess}(A)$
the location of eigenvalues within the gap of $\sigma_{\ess}(A)$ should be irrelevant. Theorem~\ref{thm:essential_left} below makes this precise.

\begin{theorem}
\label{thm:essential_left}
 Let $A$ be self-adjoint, $(a,b) \cap \sigma_\ess(A) = \emptyset$, $a \in \sigma_\ess(A)$, $\kappa \geq 0$, and $B$ bounded and symmetric satisfying $0 \leq B$ and $\lVert B \rVert < b-a$.
 Assume that
\[
 \forall x \in \bigcup_{t \in [0,1]} \ran P_{A + t B} (b) \colon \quad
 \langle x , B x \rangle \geq \kappa  \lVert x \rVert^2  .
\]
 Then
 \[
  [a + \kappa, b) \cap \sigma_\ess(A + B) \neq \emptyset.
 \]
\end{theorem}

\begin{proof}
 Define $\epsilon = b - a - \lVert B \rVert > 0$.
 We define a sequence of disjoint intervals
 \[
  I_k =
  \left( b- \epsilon + \frac{\epsilon}{2^k}, b - \epsilon + \frac{3 \epsilon}{2^{k+1}} \right) ,
  \quad
  k \in \NN.
 \]
 Note that by \eqref{eq:seelmann_1} in Lemma~\ref{lem:lemma_albrecht_s}, for all $t \in [0,1]$ we have $\sigma_\ess(A + tB) \cap (b- \epsilon, b) = \emptyset$.
 \par
 Choose $\gamma_1 \in \rho(A) \cap I_1$ and $s_1 = \min \{ \dist^{\leftarrow} (\gamma_1, \sigma(A)) / \lVert B \rVert, 1 \}$.
 We will now recursively define sequences
 \[
   (\gamma_n)_{n \in \NN} \subset \bigcup_{k \in \NN} I_k
   \quad
   \text{and}
   \quad
   (s_n)_{n \in \NN}
   \subset [0,1].
 \]
 For that purpose, we will denote by $k_n \in \NN$, the (unique) index such that $\gamma_n \in I_{k_n}$.
 If $s_n = 1$, we set $s_m = 1$ and $\gamma_{m} = \gamma_n$ for all $m > n$.
 Else, given $n \in \NN$, $\gamma_n$ and $s_n < 1$, we choose
 \[
  \begin{cases}
   \gamma_{n+1}
   = \gamma_n
   \quad
   &\text{if}
   \quad
   \sigma(A + s_n B) \cap [\sup I_{k_n + 1}, \gamma_n) = \emptyset,\\
   \gamma_{n+1}
   \in I_{k_n + 1} \cap \rho(A + s_n B)
   \quad
   &\text{else},
  \end{cases}
 \]
and set $s_{n+1} = \min \{ s_n + \dist^\leftarrow(\gamma_{n+1}, \sigma(A + s_n B)) / \lVert B \rVert, 1 \}$.
 The sequence $(s_n)_{n \in \NN}$ is monotone increasing and bounded by $1$.

 Assume that $\lim_n s_n = s < 1$.
 If there is $n_0 \in \NN$ such that $\gamma_n = \gamma_{n_0}$ for all $n \geq n_0$, then for all $n \in \NN$, $s_{n+1} - s_n$ is bounded from below by the distance between $I_{n_0}$ and $I_{n_0 + 1}$ which is a fixed positive number for all $n \geq n_0$.
 Hence, the sequence $(\gamma_n)_{n \in \NN}$ cannot become stationary.
 This implies that $A$ has infinitely many eigenvalues in $(b - \epsilon - \lVert B \rVert s, \gamma_1)$.
 This is a contradiction, since $\sigma_\ess(A) \cap [b - \epsilon - \lVert B \rVert s, \gamma_1] = \emptyset$.
 This shows $\lim_n s_n = 1$.

 By Lemma~\ref{lem:lemma_albrecht_s}, we have $\lambda_{\infty, \gamma}^\leftarrow(A + t B) = \lambda_{\infty, \tilde \gamma}^\leftarrow(A + t B)$ for all $t \in [0,1]$ and $\gamma, \tilde \gamma \in \rho(A + t B) \cap (b- \epsilon, b)$.
Given $n \in \NN$, we apply Theorem~\ref{thm:gaps_left_opt} with $\gamma = \gamma_n$, $A$ replaced by $A + s_n B$ and $B$ replaced by $(s_{n+1} - s_n) B$ and obtain
\begin{align*}
 \lambda^\leftarrow_{\infty, \gamma_n}(A + s_n B)
 &=
 \lambda^\leftarrow_{\infty, \gamma_n}(A + s_{n-1} B + (s_n - s_{n-1}) B)\\
 &\geq
 \lambda^\leftarrow_{\infty, \gamma_n}(A + s_{n-1} B) + (s_n - s_{n-1}) \kappa\\
 &=
 \lambda^\leftarrow_{\infty, \gamma_{n-1}}(A + s_{n-1} B) + (s_n - s_{n-1}) \kappa.
\end{align*}
Iteratively, we obtain for all $n \in \NN$ that $\lambda^\leftarrow_{\infty, \gamma_n}(A + s_n B) \geq \lambda^\leftarrow_{\infty, \gamma_1}(A) + s_n \kappa = a + s_n \kappa$,
which implies
 \[
  [a + s_n \kappa, a + s_n \lVert B \rVert] \cap \sigma_\ess(A + s_n B) \neq \emptyset.
 \]
 Using \eqref{eq:seelmann_2} from Lemma~\ref{lem:lemma_albrecht_s} and $0 \leq (1 - s_n) B \leq B$ this yields in particular
\[
  (a - \delta + s_n \kappa, a + \lVert B \rVert + \delta) \cap \sigma_\ess(A + B) \neq \emptyset,
  \quad
  \text{for all $\delta > 0$}
\]
and since $\sigma_\ess(A + B)$ is closed, we find
 \[
  [a + s_n \kappa, b) \cap \sigma_\ess(A + B) \neq \emptyset.
 \]
 The statement follows by taking the supremum over $n \in \NN$.
\end{proof}
Finally, combining the last theorem with Lemma~\ref{lem:lipschitz_left_edge},
we are able to give two sided Lipschitz estimates on the movement of upper edges of the essential spectrum.
\begin{corollary}
\label{cor:lifting_upper_edge}
Let $A$ be self-adjoint, $B \not = 0$ bounded, symmetric and non--negative,
$a\in \sigma_{\mathrm{ess}}(A)$,  $b > a$ such that $(a,b)\cap \sigma_{\mathrm{ess}}(A) = \emptyset$,
and  $t_0 = (b-a)/\lVert B \rVert$. Assume that
\[
 \forall x \in \bigcup_{t \in [0,1]} \ran P_{A + t B} (b) \colon \quad
 \langle x , B x \rangle \geq \kappa  \lVert x \rVert^2  .
\]
Then $f(t) = \sup \left( \sigma_{\mathrm{ess}}(A + t B) \cap (-\infty, b) \right) $
satisfies
\[
 \kappa \epsilon \leq f(t+\epsilon)- f(t) \leq \Vert B\Vert \epsilon \quad \text{ for all } t \in [0,t_0) \text{ and } \epsilon  \in [0,t_0-t).
\]
\end{corollary}

\subsection{Ordering from left to right in gaps}
\label{ssec:leftright}

Let $A$ be self-adjoint.
For $\gamma \in \rho(A) \cap \RR$ we set $\lambda_{\infty , \gamma}^\rightarrow(A) = \inf \{ \lambda > \gamma \colon \lambda \in \sigma_\mathrm{ess}(A) \}$.
If there is no essential spectrum above $\gamma$ this means $\lambda_{\infty , \gamma}^\rightarrow(A)  = \infty$,
otherwise $\lambda_{\infty , \gamma}^\rightarrow(A) $ is the lower edge of the component of $\sigma_\mathrm{ess}(A)$ above $\gamma$.
Moreover, we denote by $\{ \lambda_{k , \gamma}^\rightarrow(A) \}_{k \in \NN}$ the sequence of eigenvalues of $A$ in $(\gamma, \lambda_{\infty , \gamma}^\rightarrow(A))$,
enumerated non-decreasingly and counting multiplicities.
If there are infinitely many eigenvalues inside $(\gamma, \lambda_{\infty , \gamma}^\rightarrow(A))$ it follows that
$\lambda_{\infty , \gamma}^\rightarrow(A)= \sup_k \lambda_{k , \gamma}^\rightarrow(A)= \lim_k \lambda_{k , \gamma}^\rightarrow(A)$.
If there are exactly $N \in \NN_0$ eigenvalues inside $(\gamma, \lambda_{\infty , \gamma}^\rightarrow(A))$ then we set $\lambda_{k,\gamma}^\rightarrow (A) = \lambda_{\infty , \gamma}^\rightarrow(A)$ for all $k \in \NN$ with $k > N$.

The two main results of this subsection are the following theorems,
dealing with indefinite and non-negative perturbations, respectively.

\begin{theorem} \label{thm:gaps_right}
Let $A$ be self-adjoint, $B$ bounded and symmetric, $\gamma \in \rho (A) \cap \RR$, $\lVert B \rVert \leq \dist (\gamma , \sigma (A))/2$,
$\kappa \in \RR$, $E > \gamma$, and
\[
 \forall x \in \ran (P_{A+B} (E)) \colon \quad
 \langle x , B x \rangle \geq \kappa \lVert x \rVert^2 .
\]
Then for all $k \in \NN \cup \{\infty\}$ with $\lambda_{k , \gamma}^\rightarrow  (A + B) < E$ we have
 \[
\lambda_{k , \gamma}^\rightarrow  (A + B)
\geq
 \lambda_{k , \gamma}^\rightarrow  (A) + \kappa .
\]
\end{theorem}

If the perturbation $B$ is non-negative,
only the distance between $\gamma$ and the closest spectral value below $\gamma$ should be relevant, thus allowing to relax the condition on the norm of $B$.
Note that in contrast to the analog Theorem~\ref{thm:gaps_left_opt} in the previous subsection, we can slightly relax the positivity condition on $B$.
Also we can again compare $B$ with another operator $C$ and not just a scalar.

\begin{theorem} \label{thm:gaps_right_opt}
Let $A$ be self-adjoint, $\gamma \in \rho (A) \cap \RR$, $\kappa \in \RR$, $B, C$ bounded and symmetric satisfying $0 \leq C \leq B$, $\lVert B \rVert < \dist^\leftarrow (\gamma , \sigma (A))$, and $E > \gamma$.
\begin{enumerate}[(1)]
 \item
If \quad $ \displaystyle{\quad\quad
 \forall x \in \ran (P_{A+B} (E)) \colon \quad
 \langle x , B x \rangle \geq \kappa \lVert x \rVert^2}$\\
then for all $k \in \NN \cup \{\infty\}$ with $\lambda_{k , \gamma}^\rightarrow  (A + B) < E$ we have
 \[
\lambda_{k , \gamma}^\rightarrow  (A + B)
\geq
 \lambda_{k , \gamma}^\rightarrow  (A) + \kappa .
\]
\item If \quad $ \displaystyle{\quad\quad   \forall x \in \ran (P_{A+B} (E)) \colon \quad
 \langle x , B x \rangle \geq \langle x , C x \rangle}$\\
then for all $k \in \NN \cup \{\infty\}$ with $\lambda_{k , \gamma}^\rightarrow  (A + B) < E$ we have
 \[
\lambda_{k , \gamma}^\rightarrow  (A + B)
\geq
 \lambda_{k , \gamma}^\rightarrow  (A+C) .
\]
\end{enumerate}

\end{theorem}%
Now we turn to the proofs of the two Theorems.
As in the previous subsection,
we will apply a variational principle for eigenvalues, this time Theorem~\ref{thm:minimax_appendix}
 from the appendix \ref{sec:albrecht} written by Albrecht Seelmann.
In our context it reads as follows.

\begin{theorem}
 \label{thm:minimax_main_part}
 Let $A$ be self-adjoint, $B$ be bounded and symmetric, and $\gamma \in \RR$.
 If
 \begin{enumerate}[(i)]
  \item $\langle x,(A  - \gamma)x\rangle \leq 0$ for all $x\in \ran P_{A + B} (\gamma) \cap \cD (A)$
  and
  \item $\norm{P_{A + B}^\perp(\gamma) P_A(\gamma)}<1$,
 \end{enumerate}
 then
 \begin{equation*}
  \lambda_k(A|_{\ran P_A^\perp(\gamma)})
  = \inf_{\substack{\cL \subset \ran P_{A + B}^\perp(\gamma) \cap \cD (A) \\ \dim(\cL )=k}} \quad \sup_{\substack{x\in\cL \oplus \ran P_{A + B}(\gamma)\cap \cD (A)\\ \norm{x}=1}}
  \langle x,Ax\rangle
 \end{equation*}
 for all $k\in\NN$ with $k\le\dim(\ran P_{A + B}^\perp(\gamma))$.
\end{theorem}

We isolate the following step which is used both in the proof of Theorem~\ref{thm:gaps_right} and \ref{thm:gaps_right_opt}.

\begin{lemma}
\label{lem:calculations_for_minmax_right}
 Let $A$ be self-adjoint, $B, C$ bounded and symmetric, $\gamma \in \rho(A + B) \cap \RR$, $E > \gamma$, and $\kappa \in \RR$.
\begin{enumerate}[(1)]
 \item
If \quad $ \displaystyle{\quad\quad
 \forall x \in \ran (P_{A+B} (E)) \colon \quad
 \langle x , B x \rangle \geq \kappa \lVert x \rVert^2}$\\
then for all $k \in \NN$ with $\lambda_{k , \gamma}^\rightarrow  (A + B) < E$ we have
\begin{equation} \label{eq:Matthias}
\lambda_{k , \gamma}^\rightarrow (A+B)
\geq
\inf_{\substack{\cL \subset \ran P_{A+B}^\perp (\gamma) \cap \cD (A) \\ \dim \cL = k}}
 \quad
 \sup_{\substack{x \in \cL \oplus ( \ran P_{A+B} (\gamma) \cap \cD (A) ) \\ \lVert x \rVert = 1}}
 \langle x, A x \rangle + \kappa .
\end{equation}
\item
If $\displaystyle{ \quad \quad  \forall x \in \ran (P_{A+B} (E)) \colon \quad
 \langle x , B x \rangle \geq  \langle x , C x \rangle  }$\\
then for all $k \in \NN$ with $\lambda_{k , \gamma}^\rightarrow  (A + B) < E$ we have
\begin{equation} 
\lambda_{k , \gamma}^\rightarrow (A+B)
\geq
\inf_{\substack{\cL \subset \ran P_{A+B}^\perp (\gamma) \cap \cD (A) \\ \dim \cL = k}}
 \quad
 \sup_{\substack{x \in \cL \oplus ( \ran P_{A+B} (\gamma) \cap \cD (A) ) \\ \lVert x \rVert = 1}}
 \langle x, (A+C) x \rangle  .
\end{equation}
\end{enumerate}
\end{lemma}

\begin{proof}
Let us consider case (1).
 By assumption, we have for $\epsilon_0 = E-\lambda_{k , \gamma}^\rightarrow  (A + B)$
 \begin{align*}
 \lambda_{k,\gamma}^\rightarrow( A + B )
 & =
 \inf_{0 < \epsilon < \epsilon_0} \quad
 \sup_{\substack{x \in \ran P_{A+B} (\lambda_{k , \gamma}^\rightarrow  (A + B) + \epsilon) \cap \cD (A) \\ \lVert x \rVert = 1}}
 \langle x, (A+B) x \rangle
 \\
 & \geq
 \inf_{0 < \epsilon < \epsilon_0} \quad
 \sup_{\substack{x \in \ran P_{A+B} (\lambda_{k , \gamma}^\rightarrow  (A + B) + \epsilon) \cap \cD (A) \\ \lVert x \rVert = 1}}
 \langle x, A x \rangle + \kappa .
\end{align*}
For an self-adjoint operator $A$ we use the notation $P_A((E_1, E_2]):=\mathbf{1}_{(E_1, E_2]}(A)$.
Since
\[
\ran P_{A+B} ((\gamma, \lambda_{k , \gamma}^\rightarrow  (A + B) + \epsilon]) \cap \cD (A)
\]
is a subspace of $\ran P_{A+B}^\perp (\gamma) \cap \cD (A)$ and has dimension at least $k$ for all $\epsilon > 0$, we further estimate
\begin{align*}
&\sup_{\substack{x \in \ran  P_{A+B} (\lambda_{k , \gamma}^\rightarrow  (A + B) + \epsilon) \cap \cD (A)\\ \lVert x \rVert = 1}}
 \langle x, A x \rangle
\\
&=
 \sup_{\substack{\cL \subset \ran P_{A+B} ((\gamma, \lambda_{k , \gamma}^\rightarrow  (A + B) + \epsilon]) \cap \cD (A)\\ \dim \cL \geq k}}  \quad
 \sup_{\substack{x \in \cL \oplus (\ran P_{A+B} (\gamma) \cap \cD (A))\\ \lVert x \rVert = 1}}
 \langle x, A x \rangle
 \\
&\geq
 \sup_{\substack{\cL \subset \ran P_{A+B} ((\gamma, \lambda_{k , \gamma}^\rightarrow  (A + B) + \epsilon)]) \cap \cD (A) \\ \dim \cL = k}}  \quad
 \sup_{\substack{x \in \cL \oplus (\ran P_{A+B} (\gamma) \cap \cD (A)) \\ \lVert x \rVert = 1}}
 \langle x, A x \rangle
 \\
& \geq
 \inf_{\substack{\cL \subset \ran P_{A+B} ((\gamma, \lambda_{k , \gamma}^\rightarrow  (A + B) + \epsilon]) \cap \cD (A) \\ \dim \cL = k}}
 \quad
 \sup_{\substack{x \in \cL \oplus (\ran P_{A+B} (\gamma) \cap \cD (A)) \\ \lVert x \rVert = 1}}
 \langle x, A x \rangle
\\
& \geq
 \inf_{\substack{\cL \subset \ran P_{A+B}^\perp (\gamma) \cap \cD (A) \\ \dim \cL = k}}
 \quad
 \sup_{\substack{x \in \cL \oplus \ran P_{A+B} (\gamma) \cap \cD (A) \\ \lVert x \rVert = 1}}
 \langle x, A x \rangle .
\end{align*}
Case (2) requires only straightforward modifications of the proof.
 \qedhere
\end{proof}

We are now ready to prove Theorems~\ref{thm:gaps_right} and~\ref{thm:gaps_right_opt}

\begin{proof}[Proof of Theorem~\ref{thm:gaps_right}]
 We first consider the case $k \in \NN$. Since $\lVert B \rVert < \dist (\gamma , \sigma (A))$ we have $\gamma \in \rho (A + B)$.
 We can thus apply Lemma~\ref{lem:calculations_for_minmax_right} and obtain Ineq.~\eqref{eq:Matthias} for all $k \in \NN$ with $\lambda_{k , \gamma}^\rightarrow (A+B) < E$.
 Since $\lambda_{k , \gamma}^\rightarrow (A) = \lambda_k (A|_{\ran P_A^\perp (\gamma)})$, it now suffices to check the assumptions of Theorem~\ref{thm:minimax_main_part}.
 For all normalized $x \in \ran P_{A + B} (\gamma) \cap \cD(A)$
we have by the assumption $\lVert B \rVert \leq \dist (\gamma , \sigma (A)) / 2$ that
\begin{align*}
\langle x , (A + B) x \rangle
&\leq
\sup ( \sigma (A+B) \cap (-\infty , \gamma] )
\leq
\gamma - \dist (\gamma , \sigma (A+B))
\\
&\leq
\gamma - \dist (\gamma , \sigma (A)) + \lVert B \rVert
\leq
\gamma - \lVert B \rVert
\leq
\gamma + \left\langle x, B x \right\rangle.
\end{align*}
This shows that assumption (i) of Theorem~\ref{thm:minimax_main_part} is satisfied.
It remains to check assumption (ii) of Theorem~\ref{thm:minimax_main_part}.
This is a consequence of the Davis-Kahan $\sin 2 \Theta$ theorem \cite[Theorem~8.2]{DavisK-70}.
We apply a version given in \cite[Remark~2.9]{Seelmann-14},
and obtain
\begin{align*}
\lVert P^\perp_{A + B}(\gamma) \cdot P_A(\gamma) \rVert 
&\leq \lVert P_{A+B} (\gamma) - P_A (\gamma) \rVert \\
&\leq 
\sin \left( \frac{1}{2} \arcsin \left( 2 \frac{\lVert B \rVert}{\dist^\leftarrow(\gamma,\sigma(A))+\dist^\rightarrow(\gamma,\sigma(A))} \right) \right)\\
&\leq
\sin \left( \frac{1}{2} \arcsin \left( 2 \frac{\lVert B \rVert}{\dist(\gamma, \sigma(A))} \right) \right)
\leq
\frac{1}{\sqrt{2}}
<
1 .
\end{align*}
The case $k = \infty$ follows by taking the supremum.
\end{proof}

\begin{proof}[Proof of Theorem~\ref{thm:gaps_right_opt}]
  First we consider claim (1) and the case $k \in \NN$. Since $0 \leq B$ and $\lVert B \rVert < \dist^\leftarrow (\gamma , \sigma (A))$ we have $\gamma \in \rho (A + B)$, see \cite[Proposition~2.1]{Seelmann-17-arxiv}.
  Applying Lemma~\ref{lem:calculations_for_minmax_right}, we arrive at Ineq.~\eqref{eq:Matthias} for all $k \in \NN$ with $\lambda_{k , \gamma}^\rightarrow (A+B) < E$. Since $B$ is non-negative, assumption (i) of Theorem~\ref{thm:minimax_main_part} is satisfied.
  Assumption (ii) of the same theorem follows in a similar way as in in the proof of Theorem~\ref{thm:gaps_right}, but using the $\sin 2 \Theta$ theorem for non-negative perturbations in \cite[Theorem~1.1]{Seelmann-17-arxiv}. This shows the statement for $k \in \NN$.
\par
  The case $k = \infty$ follows by taking the supremum.
Minor modifications involving \\
 $\norm{P_{A + B}^\perp(\gamma) P_{A + C}(\gamma)}<1$ give the proof of claim (2).
\end{proof}

By mimicking the proof of Theorem~\ref{thm:essential_left}, we also find the following theorem.

\begin{theorem}
\label{thm:essential_right}
 Let $A$ be self-adjoint, $(a,b) \cap \sigma_\ess(A) = \emptyset$, $b \in \sigma_\ess(A)$, $\kappa \geq 0$, and $B$ an operator satisfying $0 \leq B$ and $\lVert B \rVert < b-a$.
 Assume that
\[
 \forall x \in \bigcup_{t \in [0,1]} \ran P_{A + t B} (b + \lVert B \rVert) \colon \quad
 \langle x , B x \rangle \geq \kappa  \lVert x \rVert^2  .
\]
 Then
 \[
  [b, b + \kappa) \cap \sigma_\ess(A + B) = \emptyset.
 \]
\end{theorem}
This allows us to describe the movement of an lower edge of a component of
the essential spectrum.
\begin{corollary}
\label{cor:lifting_lower_edge}
Let $A$ be self-adjoint, $B \not = 0$ bounded, symmetric and non--negative,
$b\in \sigma_{\mathrm{ess}}(A)$,  $b > a$ such that $(a,b)\cap \sigma_{\mathrm{ess}}(A) = \emptyset$,
and  $t_0 = (b-a)/\lVert B \rVert$. Assume that
\[
 \forall x \in \bigcup_{t \in [0,1]} \ran P_{A + t B} (b + \lVert B \rVert) \colon \quad
 \langle x , B x \rangle \geq \kappa  \lVert x \rVert^2  .
\]
Then $f(t) = \inf \left( \sigma_{\mathrm{ess}}(A + t B) \cap (a+ t \lVert B\rVert, \infty) \right) $
satisfies
\[
 \kappa \epsilon \leq f(t+\epsilon)- f(t) \leq \Vert B\Vert \epsilon \quad \text{ for all } t \in [0,t_0) \text{ and } \epsilon  \in [0,t_0-t).
\]\end{corollary}
\section{Perturbation of spectral band edges and eigenvalues of Schr\"odinger operators}
\label{sec:application}
In this section we again consider Schr\"odinger operators $H = -\Delta + V$ in $L^2 (\Gamma)$ where $\Gamma \subset \RR^d$ is $G$-admissible as in Section~\ref{sec:UCP}.
We perturb this operator by a non-negative potential $W$ which is strictly positive on the equidistributed set $S_{\delta , Z} \subset \Gamma$.
Since the perturbation $W$ is positive only on a subset, it is not immediately clear if the spectrum will be lifted at all. However, combining the results from Sections~\ref{sec:UCP} and \ref{sec:perturb}, we prove lower bounds for the lifting of eigenvalues and of the edges of the essential spectrum, see Theorem~\ref{thm:4}.
\par
Moreover, we consider the family of operators $H + tW$ with coupling constant $t \in \RR$.
As in Lemma~\ref{lem:lipschitz_left_edge} one can locally parametrize the edges of the essential spectrum of $H + t W$ as a function $t\mapsto f(t)$.
In Corollary~\ref{cor:4b} we conclude that $t\mapsto f(t)$ is strictly monotone, and provide upper and lower bounds in terms of linear functions of $t$.
\begin{theorem} \label{thm:4}
Let $G > 0$, $\Gamma \subset \RR^d$ be $G$-admissible, $\delta \in (0,G/2)$, $Z$ be a $(G,\delta)$-equidistributed sequence, $V \in L^\infty (\Gamma)$ be real-valued, and $W \in L^\infty (\Gamma)$ be real-valued such that
 \[
 W \geq \vartheta \mathbf{1}_{\equiset}
 \]
 for some $\vartheta > 0$.
Moreover, for $s \in \RR$ and with $N$ as in Corollary~\ref{cor:scaling} we set
\[
 \kappa(s)
 =
 \vartheta \sup_{\lambda \in \RR}
 \left(\frac{\delta}{G}\right)^{N \bigl(1 + G^{4/3} (\lVert V - \lambda \rVert_\infty + \lVert W \rVert_\infty)^{2/3} + G \sqrt{(s-\lambda)_+} \bigr)} .
\]

\begin{enumerate}[(a)]
 \item
 \textbf{Lifting of spectral values not exceeding $\inf \sigma_\ess(H)$.} Let $E \in \RR$.
 Then for all $k \in \NN \cup \{ \infty \}$ such that $\lambda_k(H + W) < E$, we have
 \[
  \lambda_k(H + W)
  \geq
  \lambda_k(H)
  +
  \kappa(E).
  \]
 \item
 \textbf{Lifting of eigenvalues (counted decreasingly) in gaps of $\sigma_\ess(H)$.}
 Let $\gamma \in \rho(H) \cap \RR$ and $\lVert W \rVert_\infty < \dist^\leftarrow(\gamma, \sigma(H))$.
 Then for all $k \in \NN \cup \{ \infty \}$
 \[
  \lambda_{k, \gamma}^\leftarrow(H + W)
  \geq
  \lambda_{k, \gamma}^\leftarrow(H)
  +
  \kappa(\gamma).
 \]
 \item
 \textbf{Lifting of eigenvalues (counted increasingly) in gaps of $\sigma_\ess(H)$.} \
 Let $\gamma \in \rho(H) \cap \RR$, $\lVert W \rVert_\infty < \dist^\leftarrow(\gamma, \sigma(H))$, and $E > \gamma$.
 Then for all $k \in \NN \cup \{ \infty \}$ such that $\lambda_{k, \gamma}^\rightarrow(H + W) < E$, we have
 \[
  \lambda_{k, \gamma}^\rightarrow(H + W)
  \geq
  \lambda_{k, \gamma}^\rightarrow(H)
  +
  \kappa(E)
  \]
 \item
 \textbf{Lifting of a lower edge of a gap in $\sigma_\ess(H)$.}
 Let $(a,b) \cap \sigma_\ess(H) = \emptyset$, $a \in \sigma_\ess(A)$ and assume that $\lVert W \rVert_\infty < b - a$.
 Then
 \[
  [a + \kappa(b), b) \cap \sigma_\ess(H + W) \neq \emptyset.
 \]
 \item
 \textbf{Lifting of an upper edge of a gap in $\sigma_\ess(H)$.}
 Let $(a,b) \cap \sigma_\ess(H) = \emptyset$, $b \in \sigma_\ess(A)$ and assume that $\lVert W \rVert_\infty < b - a$.
 Then
 \[
  [b, b + \kappa(b + \lVert W \rVert_\infty)) \cap \sigma_\ess(H + W) = \emptyset.
 \]
\end{enumerate}
\end{theorem}
The parameter $\kappa$ neither depends on the set $\Gamma$ nor on the choice of the equidistributed sequence and depends on the potentials $V$ and $W$ only by their $L^\infty$-norm.

\begin{proof}
 By Corollary~\ref{cor:scaling}, for all $t \in [0,1]$, $s \in \RR$ and $x \in \ran P_{H + t W}(s)$, we have
 \begin{align*}
  \left\langle x, W x \right\rangle
  \geq
  \vartheta \left\langle x, \mathbf{1}_{\equiset} x \right\rangle
  &\geq
  \vartheta \sup_{\lambda \in \RR}
  \left(\frac{\delta}{G}\right)^{N \bigl(1 + G^{4/3} \lVert V + t W - \lambda \rVert_\infty^{2/3} + G \sqrt{(s-\lambda)_+} \bigr)}
  \lVert x \rVert^2\\
  &\geq
  \vartheta \sup_{\lambda \in \RR}
  \left(\frac{\delta}{G}\right)^{N \bigl(1 + G^{4/3} (\lVert V - \lambda \rVert_\infty + \rVert W \rVert_\infty )^{2/3} + G \sqrt{(s-\lambda )_+ } \bigr)}
  \lVert x \rVert^2.
 \end{align*}
 Then the statements (a)--(e) follow by applying Lemma~\ref{lem:bottom} and Theorems~\ref{thm:gaps_left_opt}, \ref{thm:gaps_right_opt}, \ref{thm:essential_left}, \ref{thm:essential_right}, respectively.
\end{proof}

Phrasing Theorem~\ref{thm:4} in the language of Lemma~\ref{lem:lipschitz_left_edge} we immediately obtain the following corollary.

\begin{corollary} \label{cor:4b}
 Under the assumptions of Theorem~\ref{thm:4}, let $a,b \in \sigma_\ess(H)$, and $b > a$ such that $(a,b) \cap \sigma_\ess(H) = \emptyset$, and define $t_0 = (b-a)/\lVert W \rVert_\infty$.
Then, $f_\pm : (-t_0 , t_0) \to \RR$,
\begin{align*}
 f_-(t) &= \sup \left( \sigma_{\mathrm{ess}}(H + t W) \cap (-\infty, b - t_- \lVert W \rVert_\infty) \right), \\
 f_+ (t) &= \inf \left( \sigma_{\mathrm{ess}}(H + t W) \cap (a + t_+ \lVert W \rVert_\infty  , \infty) \right) ,
\end{align*}
are Lipschitz continuous with Lipschitz constant $\lVert W \rVert_\infty$ and satisfy
\[
 \epsilon \kappa \leq f_\pm (t+\epsilon)-f_\pm (t) \leq \epsilon \Vert W\Vert
\]
for all $t \in (-t_0,t_0)$ and $\epsilon >0$ such that $t+\epsilon \in (-t_0,t_0)$, where
\[
 \kappa
 =
 \vartheta \sup_{\lambda \in \RR}\left(\frac{\delta}{G}\right)^{N \bigl(1 + G^{4/3} (\lVert V - \lambda \rVert_\infty + t_0 \lVert W \rVert_\infty)^{2/3} + G \sqrt{(2b-a - \lambda)_+} \bigr)} ,
\]
and, for real $t \in \RR$, we set $t_+= \max\{0, t\}$ and $t_-= \max\{0, -t\}$.
\end{corollary}


\appendix

\section{A minimax principle in spectral gaps (Written by Albrecht Seelmann)}
\label{sec:albrecht}

In \cite{GriesemerLS-99}, Griesemer, Lewis, and Siedentop devised an abstract minimax principle for eigenvalues in spectral gaps of
perturbed self-adjoint operators. The aim of this appendix is to show that this minimax principle can be adapted to the particular
situation of bounded additive perturbations with hypotheses that tend to be easier to check in this case. In fact, in the present
context, these hypotheses can explicitly by verified by means of the Davis-Kahan $\sin2\Theta$ theorem and variants thereof,
cf.~Remark~\ref{rem:GLS} below and also the proofs of Theorems~\ref{thm:gaps_right} and~\ref{thm:gaps_right_opt} in the main part
of the present paper.

Throughout this section, we fix the following notations.
\begin{hypothesis}\label{hyp:minimax}
 Let $A$, $D$ be self-adjoint operators on a Hilbert space with the same operator domain, that is, $\Dom(A)=\Dom(D)$, and let
 $\gamma\in\RR$. Denote
 \[
  P_+:=E_A\bigl((\gamma,\infty)\bigr)\,,\quad Q_+:=E_D\bigl((\gamma,\infty)\bigr)\,,
 \]
 as well as $P_-:=1-P_+$ and $Q_-:=1-Q_+$. Moreover, let
 \[
  \cD_\pm := \ran Q_\pm \cap\Dom(D)=\ran Q_\pm \cap\Dom(A)\,.
 \]
\end{hypothesis}
Here, $\E_A$ and $\E_D$ stand for the projection-valued spectral measures for $A$ and $D$, respectively, and $\ran Q_\pm$ denotes
the range of $Q_\pm$. We have also used the short-hand notation $1$ for the identity operator.

Recall that the standard Courant minimax values for the (lower semi-bounded) part $A|_{\ran P_+}$ of $A$ associated with $\ran P_+$
are given by
\begin{equation}\label{eq:Courant}
 \lambda_n(A|_{\ran P_+})=\inf_{\substack{\fM\subset\ran P_+\cap\Dom(A)\\ \dim(\fM)=n}}
 \sup_{\substack{x\in\fM\\ \norm{x}=1}}\langle x,Ax\rangle
\end{equation}
for $n\in\NN$ with $n\le\dim(\ran P_+)$, see, e.g., \cite[Theorem~12.1]{LiebL-01} and also
\cite[Section 12.1 and Exercise 12.4.2]{Schmuedgen-12}. Here, $\langle\,\cdot\,,\,\cdot\,\rangle$ denotes the inner product of the
underlying Hilbert space.

Our main result is the following theorem.
\begin{theorem}\label{thm:minimax_appendix}
 Assume Hypothesis \ref{hyp:minimax}, and suppose, in addition, that $D=A+B$ with some bounded self-adjoint operator $B$. If
 \begin{enumerate}[(i)]
  \item $\langle x,(A-\gamma)x\rangle \le 0$ for all $x\in\cD_-$\label{thm:minimax_appendix:it1}

  and

  \item $\norm{Q_+P_-}<1$,\label{thm:minimax_appendix:it2}
 \end{enumerate}
 then
 \begin{equation}\label{eq:minimaxGLS}
  \lambda_n(A|_{\ran P_+})
  = \inf_{\substack{\fM_+\subset\cD_+\\ \dim(\fM_+)=n}} \sup_{\substack{x\in\fM_+\oplus\cD_-\\ \norm{x}=1}}
  \langle x,Ax\rangle
 \end{equation}
 for all $n\in\NN$ with $n\le\dim(\ran Q_+)$.
\end{theorem}

Note that, in contrast to \eqref{eq:Courant}, in \eqref{eq:minimaxGLS} the subspaces, over which the infimum is taken, are chosen
with respect to $D$ instead of $A$. This makes it easier to compare the minimax values for $A|_{\ran P_+}$ and $D|_{\ran Q_+}$.

The proof of Theorem~\ref{thm:minimax_appendix} relies on the following variant of the minimax principle from \cite{GriesemerLS-99}.
\begin{proposition}[{see \cite[Theorem~1]{GriesemerLS-99}}]\label{prop:GLS}
 Assume Hypothesis \ref{hyp:minimax}, and suppose that
 \begin{enumerate}[(i)]
  \item $\langle x,(A-\gamma)x\rangle\le 0$ for all $x\in\cD_-$\label{prop:GLS:it1}

  and

  \item $\ran(Q_+P_+|_{\Dom(A)})=\cD_+$.\label{prop:GLS:it2}
 \end{enumerate}

 Then,
 \[
  \lambda_n(A|_{\ran P_+})
  = \inf_{\substack{\fM_+\subset\cD_+\\ \dim(\fM_+)=n}} \sup_{\substack{x\in\fM_+\oplus\cD_-\\ \norm{x}=1}}
  \langle x,Ax\rangle
 \]
 for all $n\in\NN$ with $n\le\dim(\ran Q_+)$.

 \begin{proof}
  First, observe that under condition \eqref{prop:GLS:it1} the restriction
  \[
   Q_+|_{\ran P_+\cap\Dom(A)}\colon\ran P_+\cap\Dom(A)\to\cD_+
  \]
  automatically is injective and, hence, in view of \eqref{prop:GLS:it2}, an isomorphism. This immediately follows from Step 2 of
  the proof of \cite[Theorem~1]{GriesemerLS-99}.

  Now, for $\gamma=0$, the claim agrees with Step 1 of the proof of \cite[Theorem~1]{GriesemerLS-99} with operator domains instead
  of form domains. Taking into account that $\E_{A-\gamma}\bigl((0,\infty)\bigr)=P_+$ and $\E_{D-\gamma}\bigl((0,\infty)\bigr)=Q_+$,
  the general case $\gamma\in\RR$ is obtained from this by shift.
 \end{proof}%
\end{proposition}

Since $\ran(Q_+P_+|_{\cD_+})\subset\ran(Q_+P_+|_{\Dom(A)})$, condition \eqref{prop:GLS:it2} in Proposition~\ref{prop:GLS} is
satisfied if the stronger identity $\ran(Q_+P_+|_{\cD_+})=\cD_+$ holds. In particular, this is the case if the restriction
$Q_+P_+|_{\cD_+}\colon \cD_+\to\cD_+$ is an automorphism. The main aim of this appendix is to show that the latter is guaranteed by
condition \eqref{thm:minimax_appendix:it2} in Theorem~\ref{thm:minimax_appendix}, see Lemma~\ref{lem:minimax} below. An alternative
type of condition has been considered in \cite{GriesemerLS-99}:
\begin{remark}\label{rem:GLS}
 Condition \eqref{thm:minimax_appendix:it2} in Theorem~\ref{thm:minimax_appendix} can be replaced by
 \begin{equation}\label{eq:GLS}
  \norm{(\abs{D}+\alpha)Q_+P_-(\abs{D}+\alpha)^{-1}}<1 \quad\text{ for some }\quad \alpha>0\,,
 \end{equation}
 in the spirit of \cite[Theorem~1]{GriesemerLS-99}, cf.\ Step 2 of the corresponding proof in \cite{GriesemerLS-99}; if $D$ has a
 bounded inverse, also $\alpha=0$ can be considered here. However, in the present context of $D=A+B$ with bounded $B$, condition
 \eqref{thm:minimax_appendix:it2} in Theorem~\ref{thm:minimax_appendix} tends to be easier to check than \eqref{eq:GLS}, see, e.g.,
 \cite{BhatiaDM-83,DavisK-70,Seelmann-13-arxiv,Seelmann-14,Seelmann-16,Seelmann-17-arxiv} and the references therein.
\end{remark}

Recall that for a closed operator $\Lambda$ on a Banach space with norm $\norm{\,\cdot\,}$ its domain $\Dom(\Lambda)$ can be
equipped with the \emph{graph norm}
\[
 \norm{x}_\Lambda := \norm{x} + \norm{\Lambda x}\,,\quad x\in\Dom(\Lambda)\,,
\]
which makes $(\Dom(\Lambda),\norm{\,\cdot\,}_\Lambda)$ a Banach space.

We need the following abstract result from the author's Ph.D.\ thesis \cite{Seelmann-14-diss}.
\begin{proposition}[{\cite[Lemma~3.9 and Corollary~3.10]{Seelmann-14-diss}}]\label{prop:Seel}
 Let $\Lambda$ be a closed densely defined operator on a Banach space. Moreover, let $S$ and $K$ be bounded operators on the same
 Banach space such that $S$ maps $\Dom(\Lambda)$ into itself with
 \[
  S\Lambda x - \Lambda Sx = Kx \quad\text{ for all }\quad x\in\Dom(\Lambda)\,.
 \]
 Then:
 \begin{enumerate}[(a)]
  \item $S$ is bounded on $\Dom(\Lambda)$ with respect to the graph norm for $\Lambda$, and the corresponding spectral radius
        $r_\Lambda(S):=\lim_{n\to\infty}\norm{S^n}_\Lambda^{1/n}$ does not exceed the usual operator norm $\norm{S}$.
  \item For every scalar power series $h(z)=\sum_{k=0}^\infty a_k z^k$ with radius of convergence greater than $\norm{S}$, the
        operator $h(S)=\sum_{k=0}^\infty a_k S^k$ maps $\Dom(\Lambda)$ into itself, where the series is understood in the usual
        operator norm topology.
 \end{enumerate}

 \begin{proof}
  For the sake of completeness, we reproduce the proof from \cite{Seelmann-14-diss}.

  (a). By hypothesis, one has $\Lambda Sx=S\Lambda x-Kx$ for $x\in\Dom(\Lambda)$. Thus,
  \[
   \norm{\Lambda Sx} \le \norm{S}\,\norm{\Lambda x} + \norm{K}\,\norm{x}\,,\quad x\in\Dom(\Lambda)\,,
  \]
  and, therefore,
  \begin{equation}\label{eq:estSpecRad}
   \norm{Sx}_\Lambda = \norm{Sx} + \norm{\Lambda Sx} \le \norm{S}\,\norm{x}_\Lambda + \norm{K}\,\norm{x}\,,\quad
   x\in\Dom(\Lambda)\,.
  \end{equation}
  In particular, since $\norm{x}\le\norm{x}_\Lambda$ for $x\in\Dom(\Lambda)$, this yields that $S|_{\Dom(\Lambda)}$ is bounded with
  respect to the graph norm $\norm{\,\cdot\,}_\Lambda$, and the corresponding operator norm satisfies
  $\norm{S}_\Lambda\le \norm{S}+\norm{K}$.

  By induction, it easily follows from \eqref{eq:estSpecRad} that for all $n\in\NN$ one has
  \[
   \norm{S^n x}_\Lambda \le \norm{S}^n\, \norm{x}_\Lambda + n\norm{K}\,\norm{S}^{n-1}\norm{x}\,,\quad x\in\Dom(\Lambda)\,.
  \]
  This implies that $\norm{S^n}_\Lambda\le\norm{S}^n+n\norm{K}\norm{S}^{n-1}$ for $n\in\NN$. The corresponding spectral radius
  $r_\Lambda(S)$ can then be estimated as
  \[
   r_\Lambda(S) \le \lim_{n\to\infty} \bigl( \norm{S}^n + n\norm{K}\,\norm{S}^{n-1} \bigr)^{1/n} = \norm{S}\,,
  \]
  which proves (a).

  (b). Taking into account that the operator $\Lambda$ is closed, it follows from part (a) that the series
  $\sum_{k=0}^\infty a_k (S|_{\Dom(\Lambda)})^k$ converges in the operator norm topology with respect to the graph norm
  $\norm{\,\cdot\,}_\Lambda$. Moreover, since one has $\norm{x}\le\norm{x}_\Lambda$ for $x\in\Dom(\Lambda)$, the corresponding
  limit agrees with the restriction of $h(S)$, understood in the usual operator norm topology, to $\Dom(\Lambda)$. This proves (b).
 \end{proof}%
\end{proposition}

The following result shows that condition \eqref{thm:minimax_appendix:it2} in Theorem~\ref{thm:minimax_appendix} indeed implies
condition \eqref{prop:GLS:it2} in Proposition~\ref{prop:GLS}.
\begin{lemma}\label{lem:minimax}
 Assume Hypothesis \ref{hyp:minimax}, and suppose, in addition, that $D=A+B$ with some bounded self-adjoint operator $B$. If
 \[
  \norm{Q_+P_-} < 1\,,
 \]
 then the restriction $Q_+P_+|_{\cD_+}\colon \cD_+\to\cD_+$ is an automorphism.

 \begin{proof}
  Set $\cK:=\ran Q_+\supset\cD_+$, and define bounded operators $S,T$ on $\cK$ by
  \[
   S:=Q_+P_-|_{\cK} \quad\text{ and }\quad T:=Q_+P_+|_{\cK} = I_{\cK}-S\,.
  \]
  Since $P_+$ and $Q_+$ map $\Dom(A)=\Dom(D)$ into itself, the operator $T$ clearly maps $\cD_+\subset\Dom(A)$ into itself, that
  is, $\ran(T|_{\cD_+})\subset\cD_+$. Moreover, by hypothesis one has $\norm{S}<1$, so that $T$ has a bounded inverse. In order to
  prove that $T|_{\cD_+}\colon\cD_+\to\cD_+$ is an automorphism, it therefore suffices to show that also the inverse $T^{-1}$ maps
  $\cD_+$ into itself.

  To this end, recall that $T^{-1}$ can be written as a Neumann series,
  \begin{equation}\label{eq:neumann}
   T^{-1} = (I_\cK - S)^{-1} = \sum_{k=0}^\infty S^k\,.
  \end{equation}
  Here, along with $T$, also the operator $S$ maps $\cD_+$ into itself. Denote by $\Lambda$ the part of $D=A+B$ associated with
  $\cK$, and define the bounded operator $K$ on $\cK$ by
  \begin{equation}\label{eq:K}
   K:=\bigl(Q_+P_-B - Q_+BP_-\bigr)|_\cK\,.
  \end{equation}
  Then, one has
  \begin{equation}\label{eq:Sylvester}
   S\Lambda x - \Lambda Sx = Kx \quad\text{ for all }\quad x\in\Dom(\Lambda)=\cD_+\,.
  \end{equation}
  Indeed, for $x\in\Dom(\Lambda)=\cD_+$ one computes
  \[
   \begin{aligned}
    S\Lambda x - \Lambda Sx &= Q_+P_-(A+B)x - (A+B)Q_+P_- x\\
    &= Q_+P_-Ax + Q_+P_-Bx - (A+B)Q_+P_- x\\
    &= Q_+AP_-x + Q_+P_-Bx - Q_+(A+B)P_-x\\
    &= Kx\,.
   \end{aligned}
  \]
  In view of \eqref{eq:neumann} and $\norm{S}<1$, it now follows from Proposition~\ref{prop:Seel}\,(b) that $T^{-1}$ maps
  $\cD_+=\Dom(\Lambda)$ into itself, which completes the proof.
 \end{proof}%
\end{lemma}

We are now in position to finally prove Theorem~\ref{thm:minimax_appendix}.
\begin{proof}[Proof of Theorem~\ref{thm:minimax_appendix}]
 Since one has $\ran(Q_+P_+|_{\cD_+})\subset\ran(Q_+P_+|_{\Dom(A)})\subset\cD_+$, it follows from Lemma~\ref{lem:minimax} that
 \[
  \ran(Q_+P_+|_{\Dom(A)}) = \cD_+\,,
 \]
 so that the hypotheses of Proposition~\ref{prop:GLS} are satisfied. Applying this proposition then proves the
 claim.
\end{proof}%



\paragraph{Acknowledgments}
The authors would like to thank Denis Borisov, Michela Egidi, and Marcel Hansmann for stimulating discussions.
I.N.\ was supported in part by Croatian Science Foundation under the projects 9345 and IP-2016-06-2468.
The second and last author were partially supported by the grant VE 253/6-1
\emph{Unique continuation principles and equidistribution properties of  eigenfunctions}
of the Deutsche Forschungsgemeinschaft.
A large part of this work was done while the third author was employed at Friedrich-Schiller-Universit\"at Jena.
The authors gratefully acknowledge support by the PPP-Project \emph{The cost of controlling the heat flow in a multiscale setting}.
M.T., M.T., and I.V. thank University of Zagreb,
the third author thanks Technische Universit\"at Dortmund,
and I.N. thanks Technische Universit\"at Chemnitz and Technische Universit\"at Dortmund for their hospitality.
\par
Moreover, the authors appreciated helpful discussions with Albrecht Seelmann and express their deepest gratitude for very detailed
comments on an earlier version of the manuscript.
%
%

%
\end{document}